\documentclass[11pt]{amsart}
\usepackage{geometry}                
\usepackage{pinlabel,url}
\usepackage[all]{xy}
\usepackage{tikz}
\usepackage{graphicx}
\usepackage{amssymb,amsfonts,amsmath,latexsym}
\usepackage{epstopdf}
\usepackage[shortlabels]{enumitem}
\newtheorem{theorem}{Theorem}[section]
\newtheorem{lemma}[theorem]{Lemma}
\newtheorem{proposition}[theorem]{Proposition}
\newtheorem{corollary}[theorem]{Corollary}

\newtheorem*{claim*}{Claim}
\newtheorem{claim}{Claim}

\theoremstyle{definition}
\newtheorem{remark}[theorem]{Remark}

\newtheorem{example}{Example}

\newtheorem*{theorem-wp-bound2*}{Theorem~\ref{T : twisted Linch bound 2}}

\newcommand{\FF}{\mathcal{F}}

\newcommand{\T}{\mathcal{T}}

\newcommand{\C}{\mathcal{C}}

\newcommand{\Mod}{\mathrm{Mod}}

\newcommand{\vol}{\mathrm{vol}}
\newcommand{\BA}{\mathbb A}	
\newcommand{\BZ}{\mathbb Z}
\newcommand{\BT}{\mathbb T}

\newcommand{\union}{\cup}
\newcommand{\intersect}{\cap}
\newcommand{\boundary}{\partial}

\newcommand\tsim{\kern-.4em\sim}

\newcommand\ssm{\smallsetminus}

\newcommand{\cM}{\mathring{M}}
\newcommand{\cS}{\mathring{S}}

\newcommand{\interior}{\mathrm{int}}
\newcommand{\B}{\mathcal{B}}
\newcommand{\N}{\mathcal{N}}

\newcommand{\D}{\mathcal{D}}
\newcommand{\I}{\mathcal{I}}

\newcommand{\M}{\mathcal{M}}

\renewcommand{\SS}{\mathcal{S}}
\newcommand{\R}{\mathbb{R}}
\renewcommand{\wp}[1]{||{#1||_{wp}}}
\newcommand{\BS}{\mathbb S}	
\newcommand{\CT}{\mathcal T}
\newcommand{\Area}{\mathrm{Area}}
\newcommand{\inter}{\mathrm{int}}
\newcommand{\belt}{\mathbb B}

\DeclareMathOperator{\arccosh}{arccosh}

\title[WP length and fibered fillings]{Weil--Petersson translation length and manifolds with many fibered fillings}
\author[C. Leininger, Y.N. Minsky, J. Souto, and S.J. Taylor]{Christopher Leininger, Yair N. Minsky, Juan Souto, and Samuel J. Taylor}
\date{\today}                                           

\begin{document}
\maketitle

\begin{abstract}
We prove that any mapping torus of a pseudo-Anosov mapping class with bounded normalized
Weil-Petersson translation length contains a finite set of {\em transverse and level} closed curves, and drilling out this set of curves results in one of a finite number of cusped hyperbolic $3$--manifolds.  The number of manifolds in the finite list depends only on the bound for normalized translation length.  We also prove a complementary result that explains the necessity of removing level curves by producing
new estimates for the Weil-Petersson translation length of compositions of pseudo-Anosov mapping classes and arbitrary powers of a Dehn twist. 
\end{abstract}

\section{Introduction}

Let $S$ be a hyperbolic surface and let $\T(S)$ be its Teichm\"uller space equipped with the Weil--Petersson (WP) metric $d_{wp}$. For any mapping class $\phi$, let $\Vert \phi \Vert_{wp}$ be the translation length of $\phi$ with respect to its isometric action on $(\T(S), d_{wp})$. The focus of this article is on the structure of pseudo-Anosov homeomorphisms (on any surface) with bounded normalized WP translation length.
More precisely, let $L > 0$ and define 

\[
\Phi_{wp}(L)=  \big\{ \phi \colon S \to S \mid \phi \text{ is pA and } \sqrt{|\chi(S)|} \cdot \wp{\phi}  \le L   \big\}
\]
to be the set of pseudo-Anosov homeomorphisms on all orientable surfaces whose \emph{normalized WP translation length} is at most $L$.
For $L$ sufficiently large, $\Phi_{wp}(L)$ contains pseudo-Anosov homeomorphisms on all closed surfaces of genus $g\ge 2$.  This is a consequence of the analogous statement for {\em normalized Teichm\"uller translation length}, $|\chi(S)| \cdot \Vert \phi \Vert_T$, proved by Penner \cite{Penner-bounds}, and an inequality due to Linch \cite{Linch}; see Section~\ref{S:WP bounds}.

We will prove results constraining $\Phi_{wp}(L)$ from two directions. Theorem
\ref{T : twisted Linch bound 2} will give upper bounds on normalized WP translation
length for compositions with arbitrary powers of Dehn twists, thus showing (Corollary
\ref{C: infinitely many}) that for large enough $L$, $\Phi_{wp}(L)$ contains infinitely many
conjugacy classes in each genus.  Theorems \ref{th:main_wp} and \ref{th:structure} show
that $\Phi_{wp}(L)$ is controlled by a finite number of 3-manifolds, obtained from each
$\phi:S\to S$ by forming the mapping torus and then deleting a collection of curves
transverse to fibers or level within fibers. The level curves in particular account for
the Dehn twist phenomenon analyzed in Theorem \ref{T : twisted Linch bound 2}.

\medskip

Our first result extends Linch's inequality.
\begin{theorem} \label{T : twisted Linch bound 2}
There exists $c > 0$ so that if $\phi \colon S \to S$ is a pseudo-Anosov on a closed surface, $\alpha \subset S$ is a simple closed curve with $\tau_\alpha = \tau_\alpha(\phi) \geq 9$, and $k \in \mathbb Z$, then
\[ \Vert T_\alpha^k \circ \phi \Vert_{wp} \leq \Vert \phi \Vert_T  \sqrt{c |\chi(S)|}.\]
\end{theorem}

Here $T_\alpha$ is a Dehn twist in $\alpha$ and $\tau_\alpha(\phi)$ is the {\em twisting number} of $\phi$ about $\alpha$; see \S\ref{S:good solid tori} for definitions and \S\ref{S:bound proof} for a more precise statement.  From this theorem we obtain the following additional information about $\Phi_{wp}(L)$.

\begin{corollary} \label{C: infinitely many}   There exists $L > 0$ so that the set $\Phi_{wp}(L)$ contains infinitely many conjugacy classes of pseudo-Anosov mapping classes for every closed surface of genus $g \geq 2$.
\end{corollary}

\begin{remark} The key point of Corollary~\ref{C: infinitely many} is that the conclusion
  holds for {\bf every} closed surface of genus $g \geq 2$.  Indeed, it was already known
  that for a fixed surface one can find infinitely many conjugacy classes of pseudo-Anosov
  mapping classes with bounded WP translation distance because of the nature of the
  incompleteness of $d_{wp}$ discovered by Wolpert \cite{Wol-incomplete} and Chu
  \cite{chu}.  We also note that these statements sharply contrast the situation for the
  Teichm\"uller metric, where there are only finitely many conjugacy classes  with any
  bound on translation distance for a fixed surface; see \cite{AY,Iv3}.
\end{remark}

The idea of the proof of Corollary~\ref{C: infinitely many} from Theorem~\ref{T : twisted Linch bound 2} is as follows (see \S\ref{S:examples} for details).  We can explicitly construct a $3$--manifold $M$ that contains fibers $S_g$ of genus $g$ for all $g \geq 2$, each of which contains a fixed simple closed curve $\alpha \subset M$.  Appealing to results of Fried~\cite{fried1982flow} and Thurston~\cite{thurston1986norm}, we can find a constant $c' > 0$ so that the monodromies $\phi_g \colon S_g \to S_g$ have bounded normalized Teichm\"uller translation length $|\chi(S_g)| \Vert \phi_g \Vert_T \leq c'$ (c.f.~McMullen~\cite{mcmullen2000polynomial}).  Moreover, these can be chosen so that $\tau_\alpha(\phi_g) \geq 9$ for all $g$. Theorem~\ref{T : twisted Linch bound 2} provides a $c > 0$ so that
\[ \sqrt{ |\chi(S_g)|} \  \Vert \phi_g \circ T_\alpha^k \Vert_{wp} \leq |\chi(S_g)|
\sqrt{c} \ \Vert \phi_g \Vert_T  \leq \sqrt{c} c'.\]
For all but finitely many $k$,  $\phi_g \circ T_\alpha^k$ is pseudo-Anosov, and all such pseudo-Anosov homeomorphisms are in $\Phi_{wp}(L)$, for $L = \sqrt{c} c'$.  This construction can be carried out explicitly to produce concrete bounds, but is actually much more robust; see Corollary \ref{cor:fiber_twist}.

\bigskip

For any fixed $k \in \mathbb Z$, the mapping classes $\phi_g \circ T_\alpha^k$ in the construction just described are all monodromies of a fixed $3$--manifold $M_k$, independent of $g$.  We could alternatively describe all the manifolds $M_k$ as being obtained by an integral {\em Dehn surgery} of the single $3$--manifold $M$ along $\alpha$.  Our next result, the main theorem, states that all pseudo-Anosov homeomorphisms in $\Phi_{wp}(L)$ arise from this and a related construction.

To state the main theorem, let $\phi \colon S \to S$ be a homeomorphism and $M = M_\phi$ the mapping torus, which fibers over the circle with fiber $S$.
An embedded $1$-manifold $\C$ in $M$ is called \emph{monotonic} 
with respect to $S$ if there is a foliation of $M$ by $S$--fibers such that each component
of $\C$ is either \emph{transverse} to the foliation, or \emph{level}, i.e. embedded in some leaf. When $\C$ is monotonic, we let $\C_\ell$ be the union of level curves and $\C_t$ be the union of transverse curves. 
Note that if $M$ is fibered and $\C \subset M$ is monotonic, then $M \ssm \C_t$ is fibered and $\C_\ell$ is a collection of level curves of $M \ssm \C_t$ with respect to some fibration.

\begin{theorem} \label{th:main_wp}
Fix $L>0$. For each $(\phi\colon S\to S) \in \Phi_{wp}(L)$ there is a monotonic $1$-manifold $\C_\phi \subset M_\phi$ with respect to $S$ so that the resulting collection of $3$-manifolds 
\[
\big\{M_\phi \ssm \C_\phi :  \phi \in \Phi_{wp}(L)  \big\}
\]
is finite.
\end{theorem}

Theorem \ref{th:main_wp} is the WP analog of the result of
Farb--Leininger--Margalit \cite{farb-leininger-margalit} for pseudo-Anosovs with bounded normalized Teichm\"uller translation length. In the Teichm\"uller setting, it sufficed to remove only transverse curves. For the WP metric, removing certain level curves is necessary since integral Dehn surgery along a level curve changes the monodromy by composition with a power of a Dehn twist as in the example proving Corollary~\ref{C: infinitely many} above. 
As composing with such a power of a twist can still result in pseudo-Anosovs with bounded normalized WP translation length (Theorem~\ref{T : twisted Linch bound 2}), removing level curves is unavoidable if the resulting collection of manifolds is to be finite.

In fact, Theorem \ref{th:main_wp} is really a corollary of the following result together with work of Brock--Bromberg \cite{brock2016inflexibility} and  Kojima--McShane \cite{kojima2018normalized}.

\begin{theorem}[Many fibered fillings] \label{th:structure}
Let $\cM$ be a compact $3$-manifold whose boundary components are tori such that $\mathrm{int}(\cM)$ is hyperbolic. 
Then all \emph{sufficiently long fibered fillings} $M$ of $\cM$ have the following form:
For any fiber $S$ of $M$, there is a $1$-manifold $\C =\C_\ell \sqcup \C_t$ such that
\begin{enumerate}
\item $\cM =  M \ssm  \C$.
\item The curves $\C_t$ are transverse in $M$ with respect to $S$. So $M \ssm  \C_t$ is fibered.
\item The curves $\C_\ell$ are level in $M$ with respect to $S$.
\end{enumerate}
\end{theorem}
In this theorem, {\em sufficiently long fillings} refers to the set of Dehn fillings of the manifold $M$ whose filling slopes exclude finitely many slopes on each boundary component; see Section~\ref{sec:fibered_filling}.
Example~\ref{ex:2structures} in \S\ref{S:fiber theorem proofs} below shows that when a Dehn filling fibers in multiple ways, even though the $1$--manifold $\C$ is the same for all fibers, the decomposition $\C = \C_\ell \sqcup \C_t$ depends on the particular fiber chosen, even over a single {\em fibered face} (see \S\ref{S:Thurston-Fried}).

\subsection{Outlines}
\label{sec:outlines}

The proofs of Theorem~\ref{T : twisted Linch bound 2} and Theorem~\ref{th:main_wp} are essentially independent.  The first half of the paper is devoted to the latter, while the second half to the former. 

In Section \ref{sec:background}, we recall the definition of the Weil--Petersson metric on
Teichm\"uller space and its connection to hyperbolic volume. We also review the
Floyd--Oertel branched surfaces, which play a central role in the proof of Theorem
\ref{th:structure}. Section \ref{sec:top} then establishes a few important
facts from $3$-manifold topology. These are needed in Section \ref{sec:fibered_filling}
where Theorems \ref{th:main_wp} and \ref{th:structure} are proven using a combination of  branched surface theory and hyperbolic geometry.

\subsubsection*{Proof of Theorem \ref{th:structure} (outline)}
Let $M_\beta$ be a fibered filling of $\cM$ (the index $\beta$ is a tuple of Dehn filling
parameters on the boundary components, described in Section \ref{S:fill complexity}) and
let $S_\beta$ be a fiber. To simplify the discussion, assume that $M_\beta$, and hence
$S$, has empty boundary. 
We cut $M$ along $S_\beta$ to produce a product I-bundle $\FF_\beta \cong S_\beta \times [0,1]$.  The goal is to show, for sufficiently long fillings
$\beta$, and suitably chosen $S_\beta$ in its isotopy class, that the cores of
the filling solid tori of $M_\beta$, when intersected with $\FF_\beta$, are
vertical arcs $x\times[0,1]$ and level curves $c\times\{t\}$.

Our tool for this is a decomposition of the product $\FF_\beta$ into an {\em $I$-foliated part}
and a {\em bounded part} $X_\beta$, with these properties: 

\begin{itemize}
\item The $I$-foliated part is foliated by intervals, and contains as a subproduct the intersections with $\FF_\beta$
of the filling tubes that intersect $S_\beta$.
\item  The bounded part contains the filling tubes disjoint from $S_\beta$, which we call
  the {\em floating tubes} $V^F_\beta$. The complement $X = X_\beta \ssm V^F_\beta$ is 
 one of a {\em finite   collection} of submanifolds of $\cM$, which exists independently
 of $\beta$. 
\end{itemize}

\begin{figure}[htbp]
\begin{center}
\includegraphics[width = .9 \textwidth]{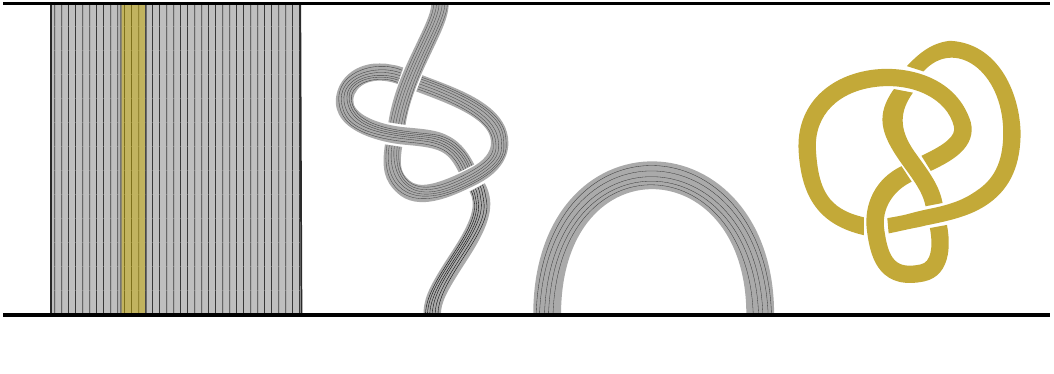}
\caption{The decomposition of the $I$-bundle $\FF_\beta$. The $I$-foliated part is
  indicated in gray and the tubes in yellow. }
\label{FXissues}
\end{center}
\end{figure}

Figure \ref{FXissues} indicates this decomposition schematically, as well as the three
basic obstructions to completing the argument: 

\begin{enumerate}
\item Knotting of the $I$-foliated part
\item Nonorientability of the $I$-foliation: a component of the $I$-foliated part whose fibers have both
  endpoints on the {\em same} component of $\boundary\FF_\beta$.
\item Knotting of the floating tubes 
\end{enumerate}

The construction comes from the Floyd-Oertel theory of branched surfaces. After an isotopy of $S_\beta$ in
$M_\beta$ to minimize a certain complexity function, $\cS_\beta = S_\beta \cap \cM$ is a properly embedded
essential surface in $\cM$ and so is fully carried by one of finitely many incompressible branched
surfaces $\B$, as discussed in Section \ref{sec:branched}. Such a branched surface has a
regular fibered neighborhood $\N$ and its complement in $\cM$ is our desired region $X$. 
After carefully arranging $S_\beta$ within $\N$ the $I$-foliated regions are obtained from
the $I$-fibration of $\N$, and $X_\beta $ is $X$ union the floating solid tori. 
Moreover,
$\boundary X$ decomposes as a vertical part $\boundary_v X$
which inherits the $I$-foliation on the corresponding part of $\boundary \N$, and
a horizontal part $\boundary_h X$ which lies in $\boundary\FF_\beta\intersect \boundary\N$.
This construction of $X_\beta$ and the foliation $\I$ is carried out in detail in Sections \ref{S:fill complexity}--\ref{S:product regions}.

The first major goal is to show that, for sufficiently long fillings $\beta$, the
$I$-foliation of $\FF_\beta$ in the complement of $X_\beta$ extends to agree, up to
isotopy, with the product foliation. The proof of this is completed in
Proposition \ref{Align foliations},
and requires the resolution of obstructions (1) and (2).

For the first, knottedness of the foliated part, we have Lemma \ref{sub products} which
says that a side-preserving embedding of $E\times[0,1]$ into $S\times[0,1]$ is unknotted
(isotopic to a standard embedding) when $E$ is not homotopically trivial. In order to
apply this we show, in Lemma \ref{no trivial components}, that for
sufficiently long fillings $\beta$ the components of the fibered part
are indeed homotopically trivial. 
The main idea, carried out in Lemma \ref{meridians in W}, is that, since $\boundary_h X$ is among a fixed finite
collection of surfaces in $\cM$, any bounded-complexity component of
$\boundary \FF_\beta \ssm \boundary_h X$ has a bounded-area homotope
in each $M_\beta$, which is used to rule out intersections with the
filling solid tori when meridians are long. 
We use this to show in Lemma \ref{no trivial components} that trivial
regions in $S \times \{0,1\}$ correspond to {\em disks of contact} for
$\B$, contradicting incompressibility of the branched surface.

The second obstruction, the possibility that components of the
foliated part have both ends on the same component of $\boundary
\FF_\beta$, is handled in Lemma \ref{transverse orientation} using the
minimal complexity assumption on the choice of $S$ within its isotopy
class.  

This allows us to prove Proposition \ref{Align foliations}, and in 
particular, we see that $X_\beta$ is itself a subproduct of
$\FF_\beta$. 

We are left with the third obstruction: showing 
that the floating solid tori are level in $X_\beta$. This is
accomplished by a reduction to a theorem of Otal, which states that
sufficiently short curves in a Kleinian surface group are level
curves. The constants in this theorem depend on the genus of the
surface, so to obtain the needed uniformity we consider a fixed fiber
surface from one fillings which, after puncturing along those solid
tori that meet the branched surface, can be embedded simultaneously in
all the fibered fillings, and represents a fiber in each of them.
This means, for sufficiently long fillings, that the cores of the
floating tori are sufficiently short to apply Otal's theorem with
respect to this fiber. A short topological argument then implies that
these curves are level with respect to all of the fiberings.

\subsubsection*{Proof of Theorem \ref{T : twisted Linch bound 2} (outline)} 
Theorem~\ref{T : twisted Linch bound 2} and its several corollaries concerning WP translation length and twisting are proven in Section \ref{sec:estimates}. The argument uses a fibered version of Dehn surgery on the mapping torus in order to twist about the curve $\alpha$. Informally, we start with the singular solv structure on the mapping torus of the pseudo-Anosov homeomorphism $\phi$, locate a solid torus foliated by flat annuli about $\alpha$, and replace this solid torus with one that performs the desired twisting while affecting the WP translation length of the new monodromy in a controlled manner. To do this, we show in Section \ref{S:constructions and estimates} that there exists a solid torus (the one used in the filling) with a \emph{leaf-wise conformal structure} that carries out the required twisting while moving a bounded distance in the WP metric. That section concludes by showing that in the singular solv structure on the mapping torus of $\phi$, one can indeed find a sufficiently large solid torus about $\alpha$ to drill out which is foliated by flat annuli. The proof of Theorem~\ref{T : twisted Linch bound 2} shows that replacing this solid torus with the one found in Section \ref{S:constructions and estimates} has the necessary effect on the WP translation length of the new monodromy. Section~\ref{sec:estimates} concludes by constructing some explicit examples (for example, proving Corollary~\ref{C: infinitely many}), as well as strengthening the construction to produce  homeomorphisms with bounded normalized WP translation length from pseudo-Anosovs over a fibered face of essentially any fixed manifold.

\bigskip

\noindent
{\bf Acknowledgements.} We would like to thank Saul Schleimer for useful conversations and
the Fall 2016 program at MSRI, where this work began. 
We also thank Yue Zhang for his comments on an earlier draft of the paper.
Leininger is partially supported by
NSF grants  DMS-1510034, DMS-1811518, as well as DMS 1107452, 1107263, 1107367 ``RNMS:
GEometric structures And Representation varieties" (the GEAR Network). Taylor is partially
supported by NSF grants DMS-1400498 and DMS-1744551. Minsky is partially supported by NSF
grant DMS-1610827.

\section{Background}
\label{sec:background}

Here we recall some basic background on the Weil-Petersson metric and branched surfaces.
\subsection{The Weil-Petersson metric on $\T(S)$}

The Teichm\"uller space $\T(S)$ of the surface $S$ is the space of marked hyperbolic structures on $S$, i.e. pairs $(X,f)$
where $X$ is a hyperbolic surface and $f \colon S \to X$ is a homeomorphism, up to the equivalence that identifies $(X,f)$ and $(Y,h)$ if there is an isometry $\Psi \colon X \to Y$ such that $\Psi$ is homotopic to $h \circ f^{-1}$.

We will primarily be interested in the Weil--Petersson (WP) metric on $\T(S)$. The cotangent space of $\T(S)$ at $X$ is naturally identified with the space ${\mathcal Q}(X)$ of (integrable) holomorphic quadratic differentials $q(z) dz^2$ on $X$, and the WP conorm on ${\mathcal Q}(X)$ is given by 
\[
\Vert\varphi\Vert_{wp}^2=\int_X \frac{\varphi \overline\varphi} {\sigma^{2}},
\]
where $\sigma = \lambda(z) |dz|$ is the hyperbolic metric on $X$. For $\mu = \mu(z) \frac{d \bar z}{d z} \in \belt(X)$, an infinitesimal $L^\infty$ Beltrami differential on $X$, representing a tangent vector to $\T(S)$ at $X$, its WP norm is defined using the pairing of Beltrami and quadratic differentials:
\[
\Vert\mu\Vert_{wp}=\max_{\varphi}\frac{\left\vert\int_X\mu\varphi\right\vert}{\Vert\varphi\Vert_{wp}}
\]
where the max is taken over all non-zero $\varphi \in {\mathcal Q}(X)$.
 The WP distance function $d_{wp}(X,Y)$ between $X, Y \in \T(S)$ is then defined in the usual way as the infimal length of paths joining $X$ and $Y$. 
(For additional background, see \cite{wolpert1987geodesic}.)

The mapping class group $\Mod(S)$ of $S$ acts on $(\T(S), d_{wp})$ by isometries, and for $\phi \in \Mod(S)$ the \emph{Weil--Petersson translation length} of $\phi$ is 
\[
\wp{\phi} = \inf _{X \in \T(S)} d_{wp}(X, \phi(X)).
\]
In the same manner we can define the Teichm\"uller translation length $\Vert \phi \Vert_T$ of $\phi$, which is equal to $\log(\lambda)$ when $\phi$ is pseudo-Anosov with dilatation $\lambda$.

\subsection{Bounds for the WP metric and volume}
\label{S:WP bounds} 

We will need the fact that the Weil-Petersson metric on the tangent space of $\T(S)$ is bounded above by the $L^2$ metric on the space of infinitesimal Beltrami differentials $\belt(S)$, with respect to the hyperbolic area form.

\begin{lemma}\label{lem wp}
Let $S$ be a closed Riemann surface uniformized by the hyperbolic metric $\sigma$. We have
$$\Vert\mu\Vert_{wp}\le
\sqrt{\int_S\vert \sigma\mu\vert^2}$$
for every $\mu \in \belt(S)$.
\end{lemma}
\begin{proof}
 
Fix $\mu \in \belt(S)$.  From Cauchy-Schwartz, we get that for every $\varphi \in {\mathcal Q}(S)$,
$$\left\vert\int_S\mu\varphi\right\vert=\left\vert\int_S\sigma\mu\varphi\sigma^{-1}\right\vert
\le\left(\int_S\vert\sigma\mu\vert^2\right)^{\frac 12}\cdot\left(\int_S\vert\varphi\sigma^{-1}\vert^2\right)^{\frac 12} = \Vert \varphi \Vert_{wp} \left(\int_S\vert\sigma\mu\vert^2\right)^{\frac 12} .$$
It thus follows that
$$\Vert\mu\Vert_{wp}=\max_\varphi\frac{\left\vert\int_S\mu\varphi\right\vert}{\Vert\varphi\Vert_{wp}}\le\max_\varphi\frac{\left(\int_S\vert\sigma\mu\vert^2\right)^{\frac 12}\cdot\Vert\varphi\Vert_{wp}}{\Vert\varphi\Vert_{wp}}=\left(\int_S\vert\sigma\mu\vert^2\right)^{\frac 12}$$
as we had claimed.
\end{proof}

The proof of this lemma includes the basic application of the Cauchy-Schwarz inequality used in the proof of the following result first observed by Linch \cite{Linch}.  We give the proof to illustrate this.

\begin{theorem} \label{th:linch}
We have $\displaystyle{\Vert v \Vert_{wp} \leq \Vert v \Vert_{T} \sqrt{\Area(S)} = \Vert v \Vert_{T} \sqrt{2 \pi |\chi(S)|}}$,
for any tangent vector $v$ to Teichm\"uller space $\CT(S)$.
\end{theorem}
\begin{proof} For any tangent vector $v$ to $\CT(S)$ at a point $[X]$, let $\mu \in \belt(X)$ be an infinitesimal Beltrami differential representing $v$ so that $\Vert v \Vert_T = \Vert \mu \Vert_\infty$.  Then by Lemma~\ref{lem wp} we have
\[ \Vert v \Vert_{wp} = \Vert \mu \Vert_{wp} \leq \sqrt{\int_S\vert \sigma\mu\vert^2} \leq \Vert \mu \Vert_\infty \sqrt{\int_S\vert \sigma \vert^2} = \Vert v \Vert_T \sqrt{\Area(S)}.
\qedhere
 \]
\end{proof}
In particular, this immediately implies that for any pseudo-Anosov $\phi \colon S \to S$, the translation lengths satisfy the following:
\[ \Vert \phi \Vert_{wp} \leq \sqrt{2 \pi |\chi(S)|} \Vert \phi \Vert_T.\]

\begin{remark}
Here is a slightly more conceptual way of explaining the inequalities in Lemma \ref{lem wp} and Theorem \ref{th:linch}.
The $L^p$ norm with respect to the hyperbolic metric can be defined for $\mu$  and for
$\varphi/\sigma^2$, where $\mu \in \belt$ is an infinitesimally trivial Beltrami differential and $\varphi \in {\mathcal Q}$ is a holomorphic quadratic
differential, because both quantities have a well-defined pointwise norm. We write these
as follows: 
$$
\Vert \mu \Vert_p = \left(\int |\mu|^p \sigma^2 \right)^{1/p}
\quad \text{ and } \quad
\Vert \varphi \Vert_p = \left(\int |\varphi/\sigma^2|^p \sigma^2 \right)^{1/p}.
$$
Note that $\Vert \varphi \Vert_1$ is exactly $\int|\varphi|$, which is the usual $L^1$ norm on ${\mathcal Q}$, and $\Vert \varphi \Vert_2$ is the Weil-Petersson conorm as
defined above. 

The usual pairing $\int\mu\varphi$ between $\belt$ and ${\mathcal Q}$ can be written
$$\langle \mu,\varphi \rangle = \int \mu(\varphi/\sigma^2) \sigma^2 $$ and 
the $L^p$ norm on ${\mathcal Q}$ induces a dual (semi)norm on $\belt$ via the pairing, namely
$$
\Vert \mu \Vert_{p*} \equiv \sup\{ |\langle \mu,\varphi \rangle| :   \Vert \varphi \Vert_p = 1\}.
$$
So $\Vert \mu \Vert_{1*}$ is exactly the Teichm\"uller norm $\Vert\mu\Vert_T$, and $\Vert\mu\Vert_{2*}$ is the
Weil-Petersson norm $\Vert\mu\Vert_{wp}$. Now
Cauchy-Schwartz applied to ${\mathcal Q}$ gives
$$
\Vert\varphi\Vert_1 \le \Vert\varphi\Vert_2 \sqrt{\Area(S)},
$$
and the definition of $\Vert\cdot\Vert_{1*}$ and $\Vert\cdot\Vert_{2*}$ above shows that
\begin{equation}
\Vert\mu\Vert_{wp} \le \Vert\mu\Vert_T \sqrt{\Area(S)}. \tag{\emph{Theorem \ref{th:linch}}}
\end{equation}
Alternatively applying Cauchy-Schwartz to the pairing $\langle \cdot,\cdot\rangle$ and the $L^2$ norms on both $\belt$ and ${\mathcal Q}$ gives
$$
|\langle \mu,\varphi \rangle| \le  \Vert\mu\Vert_2 \Vert\varphi\Vert_2,
$$
and so
\begin{equation}
\Vert\mu\Vert_{wp} = \Vert\mu\Vert_{2*} \le \Vert\mu\Vert_2.  \tag{\emph{Lemma \ref{lem wp}}}
\end{equation} 

\end{remark}

\medskip

When $\phi$ is pseudo-Anosov, the associated mapping torus $M_\phi$ is hyperbolic by Thurston's geometrization theorem for fibered manifolds \cite{thurston1998hyperbolic, otal2001hyperbolization}.
We will need the following result due to Brock--Bromberg \cite{brock2016inflexibility} and Kojima--McShane \cite{kojima2018normalized}, building on work of
Krasnov--Schlenker \cite{krasnov2008renormalized} and
Schlenker \cite{schlenker2013renormalized}, which relates the volume of $M_\phi$ to the WP translation length of $\phi$.

\begin{theorem}\label{th:vol}
Let $\phi\colon S \to S$ be pseudo-Anosov. Then 
\[
\vol(M_\phi) \le \frac{3}{2}\sqrt{2\pi |\chi(S)|)} \; \wp{\phi}.
\]
\end{theorem}

\subsection{Fibrations of a fixed $3$--manifold} 
\label{S:Thurston-Fried}

Suppose that $M$ is a compact, orientable $3$--manifold (with possibly non-empty boundary) and that $M$ fibers over the circle $f \colon M \to \BS^1 \cong \mathbb R /\mathbb Z$ with fiber a compact, connected, oriented surface $S$.  Poincar\'e-Lefschetz Duality and the deRham Theorem provide isomorphisms (for (co)homology with real coefficients), $H_2(M,\partial M) \cong H^1(M) \cong H^1_{dR}(M)$.  Via this isomorphism, the homology class of the fiber $S$ is identified with the cohomology class represented by the closed $1$--form $\omega = f^*(dt)$, where $dt$ is a nonwhere zero $1$--form defining the orientation on $\BS^1$.  Note that $\omega$ is nowhere zero, as is the restriction to $\partial M$, and there is a neighborhood $U \in H^1(M;\R)$ of $[\omega]$ so that every element is represented (in deRham cohomology) by such a closed $1$--form.  Any primitive integral class $\omega'$ representing a class in $\R_+ U$ can be integrated to define another fibration $f' \colon M \to \BS^1$ whose fiber $S'$ is identified with $[\omega']$ (via the above isomorphism); see \cite{Tischler}.

This construction gives rise to infinitely many fibrations of $M$ as long as $b_1(M) > 1$, which happens precisely when the monodromy $\phi \colon S \to S$ has a nontrivial fixed cohomology class.  Indeed, the subspace of $H^1(S)$ fixed by $\phi$ is precisely the image of the homomorphism $H^1(M) \to H^1(S)$ induced by inclusion, with the kernel generated by the dual of $[S]$ (because these are the classes in $H^1(S)$ that extend to $H^1(M)$).

Thurston proved that the maximal connected neighborhood $U$ of $[S] = [\omega]$ as in the previous paragraph has a particularly nice description.  To state his result, we recall that in \cite{thurston1986norm}, Thurston constructs a norm $\mathfrak n$ on $H^1(M)$, the {\em Thurston norm}, when $M$ is irreducible and atoroidal, so that the unit ball ${\bf B}$ of $\mathfrak n$ is a polyhedron.

\begin{theorem} [Thurston] \label{T:fibered face} If $S$ is a fiber of the compact, orientable, irreducible, atoroidal $3$--manifold $M$, then there is a top-dimensional face $F$ of ${\bf B}$ so that $[S] \in \R_+ \mathrm{int}(F)$ and every element of $\R_+ \mathrm{int}(F)$ is represented by a closed $1$--form which is nowhere vanishing on $M$ or $\partial M$.  Moreover, any primitive integral point of $\R_+ \mathrm{int}(F)$ determines a fibration of $M$ with fiber $S'$, and $\mathfrak n([S']) = -\chi(S')$.
\end{theorem}

A face $F$ of ${\bf B}$ as in this theorem is called a {\em fibered face}.  Note that the restriction of $\mathfrak n$ to $\R_+ \mathrm{int}(F)$ is linear (since $F$ is a face of ${\bf B}$).  In fact, $\mathfrak n$ is given by pairing with the negative of the Euler class of the $2$--plane bundle tangent to the foliation of $M$ by fibers $S$ of $f \colon M \to \BS^1$.
 
A fibered manifold $M$ with fiber $S$ of negative Euler characteristic is atoroidal if and only if the monodromy $\phi \colon S \to S$ is isotopic to a pseudo-Anosov homeomorphism.  In this case, Fried showed that the Teichm\"uller translation length (which is also equal to the topological entropy of the pseudo-Anosov homeomorphism) extends to a nice function on the cones over interiors of fibered faces; see \cite[Theorem F]{fried1982flow}.

\begin{theorem} [Fried] \label{T:entropy extension} For any compact, orientable, atoroidal manifold $M$ and fibered face $F$ of the Thurston norm ball ${\bf B}$, there is a continuous, convex, homogeneous function, $\mathfrak h \colon \R_+ \mathrm{int}(F) \to \R_+$, homogeneous of degree $-1$, such that if $S$ is a fiber of a fibration of $M$ with $[S] \in \R_+ \mathrm{int}(F)$ and monodromy $\phi \colon S \to S$, then $\Vert \phi \Vert_T = \mathfrak h(u)$.
\end{theorem}

These two theorems provide examples of pseudo-Anosovs with small Teichm\"uller translation length illustrating Penner's upper bound as follows (see \cite{mcmullen2000polynomial}).  Note that the product of the two function $\mathfrak n(\cdot) \mathfrak h(\cdot)$ is continuous and constant on rays.  In particular, if $K \subset \mathrm{int}(F)$ is any compact subset of the interior of a fibered face $F$, there is a constant $L_K$ that bounds the value of $\mathfrak n \mathfrak h$ on $\mathbb R_+ K$.  For any primitive integral point in $\mathbb R_+K$ representing a fiber with monodromy $\phi \colon S \to S$, we have
\[ |\chi(S)| \Vert \phi \Vert_T = \mathfrak n([S]) \mathfrak h([S]) \leq L_K.\]
In particular, supposing there are surfaces of all genera at least $2$ which are fibers representing elements in $\mathbb R_+ K$ (see e.g.~the proof of Corollary~\ref{C:specific infinitely many} below), then one finds examples of pseudo-Anosov homeomorphisms on every closed surface of genus at least $2$ in $\Psi_T(L_K)$.

\subsection{Incompressible branched surfaces}
\label{sec:branched}

Next, we recall the construction of the incompressible branched surfaces of Floyd--Oertel \cite{floyd1984incompressible}. Our discussion closely follows that of  Oertel \cite[Section 4]{oertel1984incompressible}), except that, as in Tollefson--Wang \cite[Section 6]{tollefson1996taut}, we describe the construction of these branched surfaces in terms of normal surfaces to a triangulation, rather than a handle decomposition. For background on normal surfaces, see \cite{haken1961theorie,jaco1984algorithm, tollefson1996taut}.

Let $M$ be a Haken $3$-manifold with incompressible boundary and a triangulation ${\bf t}$. 
The \emph{weight} $w(S)$ of a properly embedded surface $S$ in general position with ${\bf t}^{(1)}$ is the number of points in $S \cap {\bf t}^{(1)}$.  We recall that the minimal weight for $S$ within its isotopy class can be realized by a \emph{normal} representative since a minimal weight $S$ can be isotoped to be normal rel ${\bf t}^{(1)}$ \cite{jaco1989pl}.

For each incompressible normal surface $S$ with minimal weight in its
isotopy class, there is a \emph{branched surface fibered neighborhood}
$\widetilde \N_S$ produced as follows: $\widetilde \N_S$ is the union
of thickenings of the normal disks appearing in $S$ together with all
$3$-balls lying between two thickened normal disks of the same type.  
We choose compatible product structures on these thickened disks and $3$-balls so that the $I$--fibers (intervals) agree on the boundary of each tetrahedron. Hence, $\widetilde \N_S$ is foliated by $I$--fibers.
The corresponding \emph{branched surface} $\widetilde \B_S$ is the
2-complex obtained from $\widetilde \N_S$ by collapsing each of the
$I$--fibers to a point. (Usually one thinks of the branched surface as
properly embedded in $M$ with $\N$ a regular neighborhood of it; this won't be
crucial for us as we will work explicitly with the fibered
neighborhood itself.)
See Figure \ref{fig:normal}.

\begin{figure}[htbp]
\begin{center}
\includegraphics[ width = 1 \textwidth]{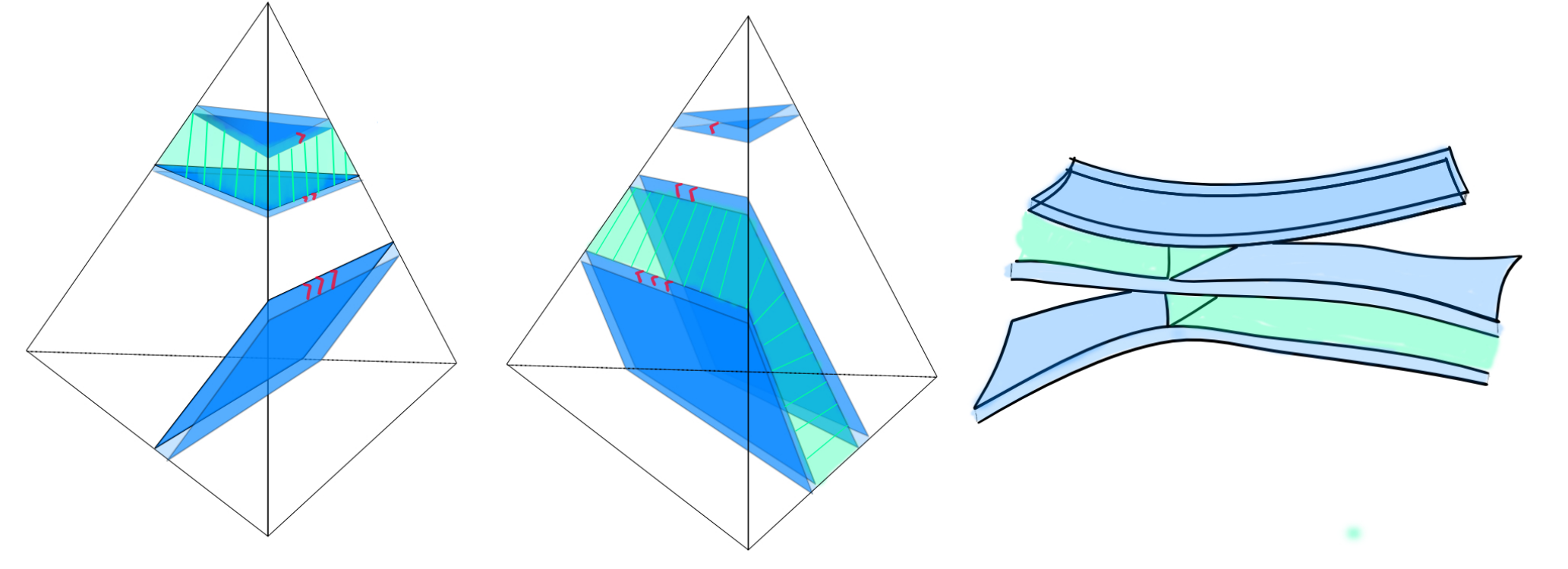}
\caption{Thickened normal disks (blue), fibered $3$-balls (green), and the branched surface fibered neighborhood.}
\label{fig:normal}
\end{center}
\end{figure}

For any branched surface fibered neighborhood $\N$, its boundary 
decomposes into a union of three subsurfaces, the horizontal boundary
$\partial_h \N$, the vertical boundary $\partial_v \N$, and $\N \cap
\partial M$. While $\partial_v \N$ is foliated by subintervals of
fibers, each $I$--fiber of $\N$ meets $\partial_h \N$ at its
endpoints. A surface $F$ in $M$ is {\em carried by $\B$} if it can be
properly isotoped into the interior of a fibered neighborhood $\N = \N(\B)$ of
$\B$ so that it intersects each fiber transversely. It is {\em fully
  carried} if in addition it has nonempty intersection with each
fiber. For example, $S$ is fully carried by $\widetilde N_S$ by
construction. 

Floyd--Oertel define a branched surface $\B$ to be \emph{incompressible} if 
\begin{enumerate}[{\em (i)}]
\item there are no disks (or half-disks) of contact,
\item there are no complementary monogons, and
\item $\partial_h \N(\B)$ is incompressible and
  $\partial$-incompressible in $M\ssm\N(\B)$.
\end{enumerate}
Here, a disk of contact is a disk $D \subset \N$ that is transverse to
the $I$--fibers of $\N$ and $\partial D \subset
\mathrm{int}(\partial_v\N)$. A half-disk of contact is a disk
$D \subset \N$ that is transverse to the fibers with $\partial D = \alpha \cup \beta$,
where $\alpha \ssm \partial \alpha \subset \mathrm{int}(\partial_v\N)$ and $\beta \subset \partial M$
are arcs and $\alpha \cap \beta = S^0$.
A complementary monogon is a disk $D \subset M \ssm \mathrm{int}(\N)$ with $\partial D = D \cap \N = \alpha \cup \beta$ where $\alpha \subset \partial_v \N$ is an $I$--fiber and $\beta \subset \partial_h \N$. Floyd--Oertel show \cite[Theorem 2]{floyd1984incompressible} that if $\B$ is an incompressible branched surface, then any surface fully carried by $\B$ is incompressible and boundary incompressible.

Unfortunately the branched surface $\widetilde \B_S$ constructed above
may have many disks of contact and therefore is not
incompressible. However, Floyd--Oertel show that such a disk of
contact $D$ may be removed by deleting from $\widetilde \N_S$ a
fibered neighborhood of $D$ in $\widetilde \N_S$, thereby producing a
new branched surface fibered neighborhood in which $S$ is fully carried. 
They prove that by applying this operation finitely 
many times, one can produce a fibered neighborhood 
$\N_S$ of an incompressible
branched surface $\B_S$ \cite[Proposition 3]{floyd1984incompressible}. 
  See also \cite[Lemma
  4.6]{oertel1984incompressible} where the construction of $\N_S$ is
done more systematically.
 We denote the corresponding
branched surface by $\B_S$ and also write $\N_S = \N(\B_S)$. 

We say that a branched surface (and its fibered neighborhod) obtained
in this way is {\em adapted to the triangulation ${\bf t}$}, and record
two important properties of the construction:

\begin{itemize}
\item Each component of $\partial_v \widetilde \N_S$ contains a subarc of the $1$-skeleton
  of ${\bf t}$ as a fiber of $\partial_v \widetilde \N_S$. Since 
eliminating a disk of contact is done by deleting a fibered neighborhood of the disk, the property that each component of $\partial_v \N$ meet the $1$-skeleton is preserved.
\item The normal surface $S$, which was assumed to have minimal weight in its isotopy class, is contained in the branched surface fibered neighborhood $\N_S$, where it intersects each fiber transversely.
\end{itemize}

By Floyd--Oertel \cite[Theorem 1]{floyd1984incompressible}  (see also \cite[Theorem 3]{oertel1984incompressible}) this procedure produces a \emph{finite} collection
of properly embedded branched surfaces
$\B_1, \ldots, \B_n$ in $M$ such that $(a)$ any surface fully carried
by one of the $\B_i$ is incompressible and boundary incompressible,
and $(b)$ every incompressible and boundary incompressible orientable
surface is  fully carried by some $\B_i$. In
particular, the branched surface $\B_S$ constructed from the surface
$S$ appears up to isotopy as one of the $\B_i$ in this list.
(That fact that $\{ \widetilde \N_S \}$ forms a finite set of branched surfaces is obvious since there are only finitely many choices for normal disks in the construction. Showing that $\{\N_S \}$ is finite is more difficult and requires showing that one only needs to consider least weight disks of contact, of which there are finitely many.)

We summarize all of this in the following statement:
\begin{proposition} [Floyd-Oertel] \label{prop:branched finiteness}
Let $M$ be a compact Haken 3-manifold with a triangulation 
${\bf t}$. There is a finite collection
$\B_1,\ldots,\B_n$ of incompressible properly embedded branched
surfaces, so that
\begin{itemize}
  \item Each fibered neighborhood $\N(\B_i)$ is adapted to $\bf t$,
  and in particular every component of $\boundary_v\N(\B_i)$ has a
  fiber which is a subarc of an edge of ${\bf t}^{(1)}$. 
  \item Every properly embedded incompressible
boundary-incompressible surface $S$ in $M$ is properly isotopic to a surface
$S'$ fully carried in  one of the $\B_i$, which realizes the minimum weight
with respect to ${\bf t}$ in the isotopy class of $S$.
\end{itemize}
\end{proposition}

\section{Topological preliminaries} 
\label{sec:top}

This section covers a number of fairly basic results from $3$-manifold topology that we will need for the proof of Theorem \ref{th:structure}.

\subsection{Embeddings in products} 
\label{sec:wald}

We first require the following lemma, which follows easily from work
of Waldhausen \cite{waldhausen1968irreducible}. 

If $S$ is a compact surface we say that $(W,\partial_b W) \subset
(S,\partial S)$ is a {\em subsurface with corners} (or just {\em
  subsurface} for short) if $W$ is a compact 2-manifold and $\partial_b W =
\partial W \intersect \partial S$ is a compact 1-submanifold of
$\partial W$. The endpoints of $\partial_b W$ are the corners and
$\partial W$ minus the corners is the {\em smooth} part of $\partial
W$.  We denote the closure of $\partial W \ssm \partial_b W$ as
$\boundary' W$. 

We call $(W,\partial_b W)$ (or just $W$) \emph{trivial} if it
is contained in a disk $D$ whose intersection with $\boundary S$ is empty or a single arc.
Note that if $W$ is not trivial, then either the image of $\pi_1 W$ in $\pi_1 S$ is nontrivial, or there is an essential arc of $S$ contained in  $W$.

Define a modified Euler characteristic as
$$\hat\chi(W,\partial_b W) = \chi(W) - {1\over 2} n$$
where $n$ denotes the number of arc components of $\partial_bW$ (see
e.g. \cite{casson-bleiler}).

\begin{lemma}\label{sub products}
Suppose that $(W, \partial_b W) \subset (S, \partial S)$ is a
subsurface with corners and that the restriction 
to each component of $W$ is not trivial.
 Let $Y = W\times[0,1]$ and suppose we have an embedding  of quadruples
  $$
  F:(Y,W\times\{0\},W\times\{1\}, \partial_b W \times [0,1]) \hookrightarrow
  (S\times[0,1],S\times\{0\},S\times\{1\}, \partial S \times [0,1]),
  $$
where $W \times \{0\} \to S \times \{0\}$ is given by the inclusion map.

 Then $F$ is isotopic, as a map of quadruples, to the
 inclusion map of $Y$. 
\end{lemma}

Note that the nontriviality condition is necessary -- consider a
knotted 1-handle attached to two disks in $S\times\{0\}$ and
$S\times\{1\}$.

\begin{proof} 
By applying a preliminary isotopy of $S \times [0,1]$ supported on a
neighborhood of $\partial S \times [0,1]$, we may assume that the
image of each arc of $\partial \partial' W \times [0,1]$ is vertical
in $\partial S \times [0,1]$.   

If a null-homotopic closed curve $\gamma$ of $\boundary' W$ bounds a
disk $E$ in $S$ then the image of $\boundary(E\times[0,1])$ bounds a
ball by irreducibility and $F$ can be extended over
$E\times[0,1]$. Thus we may assume that each component of $W$ injects
on $\pi_1$. Similarly if $\alpha$ is an arc component of $\boundary'
W$ that cobounds a disk $E$ with an arc $\beta\subset \partial S$, we
can again extend over $E\times[0,1]$. 
Hence, we may assume that $\partial' W$ consists of homotopically essential curves and
essential proper arcs.

Thus each component of $F(\partial' W \times [0,1])$ is either an incompressible annulus in $S \times [0,1]$ or a disk which meets $\partial S \times [0,1]$ in two vertical arcs.
Hence (see for example \cite[Lemma 3.4]{waldhausen1968irreducible})
$F(\partial'W\times[0,1])$ is isotopic to $\partial' W \times [0,1]$, and
we may adjust $F$ by an isotopy that is constant on $W \times \{0\}$
to obtain a map that is the identity on $\partial'W\times[0,1]$. 

One further isotopy, constant on $W \times \{0\}$ and supported on a
neighborhood of $\boundary_b W\times[0,1]$, yields a map (which
we still call $F$) that is the identity on all of $\partial
W\times[0,1]$. 

Now $F$ is the identity on $W\times\{0\} \union \boundary
  W\times[0,1]$, and it follows that $F(W\times\{1\}) = W\times\{1\}$,
    and from this that $F(Y)=Y$. 
Hence, by 
\cite[Lemma 3.5]{waldhausen1968irreducible}, the homeomorphism
$F^{-1} \colon Y \to Y$ is isotopic to the identity via an isotopy
that is constant on $W \times \{0\} \cup \partial W \times
[0,1]$. Precomposing $F$ with this isotopy gives the required isotopy
from $F$ to the identity and completes the proof. 
\end{proof}

\subsection{Level curves}
\label{sec:level}
Let $Y=Z\times I$ where $Z$ is a surface and $I$ an interval. Let $C$
be a disjoint closed union of
simple loops in $\interior(Y)$. We say $C$ is \emph{level} with respect to this
product structure on $Y$ if it is isotopic to a union of the form $\bigcup c_i\times \{t_i\}$
where $c_i$ are loops in $Z$ and $t_i\in [0,1]$. 

\begin{lemma}\label{straighten in product}
Let $Y=Z\times[0,1]$ where $Z$ is a surface. Let $C$ be a finite union of
simple loops in $Y$. Then $C$ is level if and only if there is an
ordering $c_1,\ldots,c_n$ of its components so that each $c_i$ is
isotopic to $Z\times\{0\}$ in the complement of $c_{i+1}\union\cdots\union c_n$.
\end{lemma}
The proof is an easy induction.

\begin{lemma}\label{different straightenings}
Let $S$ be  a compact orientable surface and let
 $C$ be a level disjoint closed union of essential curves in $S\times \R$.
 Let $T_1$ and
$T_2$ be surfaces isotopic to $S\times \{0\}$ rel $\partial
S\times\R$,  and disjoint  from each
other and from $C$.  If $X$ is the region bounded by $T_1$ and $T_2$ (so $X \cong S \times [0,1]$),
then $C\cap X$ is level in $X$.

Furthermore, if $C \subset Y \subset X$ where $Y \subset X$ is a subproduct, then the isotopy can be chosen to be supported in $Y$---that is, the image of $C$ remains in $Y$ throughout the isotopy.
\end{lemma}
By saying that $Y$ is a {\em subproduct of $X$} we mean that there is a subsurface $R \subset S$ and a homeomorphism of pairs $(X,Y) \cong (S \times [0,1],R \times [0,1])$.

\begin{proof} By applying an (ambient) isotopy, we may assume that the components of $C$ have the form $c \times \{t\}$ for some $t \in \R$ with respect to the given product structure on $S \times \R$.

We first prove this in the case that $T_2$ is a level surface $S\times\{t\}$.  We may assume $T_1$ lies above $T_2$. 
Let $c$ be a component of $C\cap X$ of minimal height with respect to the product structure $S \times \R$. The vertical
annulus $A$ taking $c$ to $T_2$ is disjoint from $C \setminus c$, but might intersect $T_1$.  An innermost
region of $A\intersect X$ or $A\intersect X^c$ with boundary in $T_1$ must be a disk or
annulus, and an exchange move will simplify the intersection reducing the number of components of intersection. 
Thus an annulus from $c$ to $T_2$ meeting $T_2$ only in one boundary component, disjoint from $C \setminus c$, and intersecting $T_1$ minimally, will in fact be disjoint from $T_1$, hence contained in $X$.  This gives an isotopy from $c$ to $T_2$ avoiding the remaining
curves. Proceeding by induction on the heights of components of $C$,
we can now apply Lemma \ref{straighten in product} to conclude that
$C\cap X$ is level in $X$.

Now if $T_1$ and $T_2$ are arbitrary, let $T_3 = S\times\{t\}$ be a
level surface below both. Let $X_{ij}$ be the region bounded by $T_i$
and $T_j$ and suppose $X_{13}$ contains $T_2$. Then the previous case
applies to $X_{13}$, and after choosing a product structure on
$X_{13}$ in which the top and bottom are level surfaces, we can again
apply the previous case to $X_{12}$. This concludes the
proof of the first claim.

Finally, suppose $C \subset Y$, with $(X,Y) \cong (S \times [0,1],R \times [0,1])$.  If at any stage of the isotopy of the $i^{th}$ component the corresponding annulus passes outside $Y \cong R \times [0,1]$, and hence through $\partial R \times [0,1]$, another innermost region exchange move argument implies that we can replace the annulus with one having fewer intersections with  $\partial R \times [0,1]$.  In particular, such an annulus with the fewest intersections with $\partial R \times [0,1]$ will be disjoint, and hence the isotopy will be supported in $Y$.
\end{proof}

\subsection{Minimal intersections}

Finally, we will need a few facts regarding intersections of surfaces and $1$--manifolds.  One such fact is the following proposition which roughly states that if $S$ is an embedded surface in $M$ which minimizes intersection with some $1$-manifold $C$ in its isotopy class, then one cannot \emph{homotope} $S$ to reduce intersections with $C$.

\begin{proposition}\label{prop:FHS}
Let $M$ be a compact, irreducible, orientable $3$-manifold and $C \subset M$ is a closed, embedded $1$--manifold.  
Suppose $S \subset M$ is an incompressible, boundary incompressible, properly embedded surface that minimizes the cardinality $|S \cap C|$ in its proper isotopy class.  
If $f \colon S \to M$ is any piecewise smooth map which is properly {\em homotopic} to the inclusion of $S$ and transverse to $C$, then $|f^{-1}(C)| \geq |S \cap C|$.
\end{proposition}

\begin{proof} 
Suppose $f \colon S \to M$ is any piecewise smooth map properly homotopic to the inclusion of $S$ into $M$, and which is transverse to $C$.  Set $k = |f^{-1}(C)|$.  
Let $V_C \cong C \times D^2$ be a small tubular neighborhood of $C$.  Adjusting $f$ by a proper homotopy, we may assume that $f(S) \cap V_C$ is a union of $k$ disjoint meridian disks.
Next let $N \cong \partial M \times [0,1)$ be a tubular neighborhood of the boundary of
  $M$ whose closure is disjoint from $V_C$.  Using the neighborhood $N$ and the homotopy
  from the inclusion of $S$ to $f$, we can find a further homotopy so that $f|_{\partial
    S}$ is the inclusion $\partial S \to \partial M$ (and is hence an embedding) and so that $f(S) \cap V_C$ is still a union of $k$ disjoint meridian disks.

Next, choose any Riemannian metric on $M$ with the following properties:
\begin{enumerate}
\item the restriction to $V_C \cong C \times D^2$ is a product metric,
\item the restriction to $N \cong \partial M \times [0,1)$ is a product metric,
\item the area of each meridian disk of $V_C$ is some number $A >0$, and 
\item $\Area(f) \leq (k+\frac14)A$.
\end{enumerate}
It is straightforward to construct such a metric.

Now let $h \colon S \to M$ be a least area surface properly homotopic to the inclusion of $S$ into $M$ rel $\partial S$.  Such an $h$ exists by \cite{SacksUhlenbeck,SchoenYau} in the closed case and \cite{Lemaire,HassScott} in the case of non-empty boundary.  Note that $\Area(h) \leq (k+\frac14)A$.  

According to \cite{FHS}, $h$ is an embedding, or in the case that $S$ is closed, possibly a double cover an embedded non-orientable surface with a tubular neighborhood that is a twisted $I$--bundle over the non-orientable surface (see Section 7 of \cite{FHS} for the case $\partial S \neq \emptyset$).  In the latter case, we may replace $h$ with an embedding at the expense of an arbitrarily small increase in the area.  In either case, by a further small isotopy if necessary, we can also assume that $h$ is transverse to $\partial V_C$ and $\Area(h) \leq (k + \frac12)A$.  The embedding $h$ is properly isotopic to the inclusion of $S$ into $M$ since $M$ is irreducible \cite[Corollary 5.5]{waldhausen1968irreducible}.

Consider any component $W$ of $h(S) \cap V_C$ and suppose $V \cong S^1 \times D^2$ is the component containing $W$.

Note that the projection  $V \to D^2$ onto the second factor is an isomorphism
$H_2(V,\partial V) \to H_2(D^2,\partial D^2)$, where the latter group is $\BZ[D^2]$. Thus
we can define $\deg(W)$ to be the integer such that
\[ [W] = \deg(W) [D^2] \in H_2(V,\partial V).\]
This is equal to the topological degree of the projection $(W,\partial W) \to
(D^2,\partial D^2)$, as well as the topological degree
of the map $\partial W \to \partial D^2$. 

Now consider the metric on $D^2$ for which the projection $V \to D^2$ is a Riemannian
submersion, and in particular a contraction. We thus have
\[ \Area(W) \geq |\deg(W)| A.\]
Setting
\[ d(h) = \sum_{W \subset h(S) \cap V_C} |\deg(W)| \]
where the sum is over all components of intersection, 
and combining with our area bound on $h$, we have
\[ d(h) A \leq \Area(h) \leq (k+\frac12) A.\]
Since $d(h)$ is an integer, we have $d(h) \leq k$.  The next claim essentially completes the proof.
\begin{claim}  The map $h$ is properly isotopic to an embedding $h'$, transverse to $\partial V_C$ with $d(h') \leq d(h)$ so that $h'(S) \cap V_C$ is a union of meridian disks.
\end{claim}
Assuming the claim, it follows that $|h'(S) \cap C|$ is the number of meridian disks of intersection $h'(S) \cap V_C$, which is $d(h') \leq d(h) \leq k$.  Since $S$ was assumed to meet $C$ in the fewest number of points in the proper isotopy class, and since $h'$ is properly isotopic to $h$, and hence also the inclusion of $S$, it follows that
\[ |S \cap C| \leq d(h') \leq k = |f^{-1}(C)|,\]
as required.  This completes the proof of the proposition.

\begin{proof}[Proof of Claim]
  Recall first a lemma of Thurston 
\cite[Lemma 1]{thurston1986norm}, that if a properly embedded surface in a 3-manifold $V$ represents a
$k$-th multiple of a homology class in $H_2(V,\partial V)$, then the surface has at least
$k$ components. Thus for our connected components $W$ of $h(S)\cap V_C$, we have
$$|\deg(W)| \leq 1.$$

To avoid cluttering the notation, we write $S_0 = h(S)$, $d(S_0) = d(h)$.  Throughout the proof, we repeatedly replace $S_0$ with isotopic images which we again denote $S_0$.  We first isotope $S_0$ to remove trivial intersections with $\partial V_C$. That is, if some component of $S_0 \cap \partial V_C$ bounds a disk $D$ in $\partial V_C$, then it also bounds a disk $E$ in $S_0$ by incompressibility. Suppose that $D$ is innermost on $\partial V_C$, and consider the disk swap that replaces the disk $E$ in $S_0$ with $D$.  (This can be done via isotopy of $S_0$ in $M$ since $D \cup E$ bounds a ball.) This may reduce the number of intersections with $V_C$, but it does not increase $d(S_0)$. To see this, note that if $W$ was the component of $S_0 \cap V_C$ whose boundary contained $\partial D$, then since $\partial D \to \partial D^2$ has degree $0$, the degree of $W$ is unaffected by capping this boundary component off with the disk $D$. After pushing $D$ slightly into $V_C$, we have a new surface isotopic to $S_0$ with fewer trivial intersections with $\partial V_C$ and no greater degree.

Hence, we may suppose that $S_0$ does not meet $\partial V_C$ in trivial circles. We next isotope $S_0$ so that each component of $S_0 \cap V_C$ is incompressible: If $D$ is a compressing disk for some component $W$ of $S_0 \cap V_C$, then again suppose that $D$ is innermost in the sense that $S_0$ does not meet the interior of $D$. Incompressibility of $S_0$ in $M$ implies that there is a disk $E \subset S_0$ with $\partial E = \partial D$ and we may assume that $E$ is not contained in $V_C$. 

Let $W$ be the component of $S_0 \cap V_C$ meeting $D$. Compressing along $D$ results in two surfaces $W_1$ and $W_2$ and we call $S_0'$ the surface isotopic to $S_0$ obtained by replacing $E$ with $D$ (again possible since $D \cup E$ bounds a ball).  Label so that $W_1$ is a component of $S_0' \cap V_C$.
Let $\mathcal{Y}$ be the components of $S_0 \cap V_C$ contained in $S_0 \ssm E$ and $\mathcal{Z}$ the components of $S_0 \cap V_C$ contained in $E$. Then
\[
d(S_0) = \sum_{Y \in \mathcal{Y}} |\deg(Y)| + |\deg(W)| + \sum_{Z \in \mathcal{Z}} |\deg(Z)|
\]
and 
\[
d(S_0') = \sum_{Y \in \mathcal{Y}} |\deg(Y)| + |\deg(W_1)|.
\]
We must show that $d(S_0') \leq d(S_0)$.  
Since $|\deg(W_1)| $ and $|\deg(W)|$ are both at most 1, 
it suffices to prove that there is at least one $Z \in \mathcal{Z}$ with $|\deg(Z)| = 1$. 

Let $E'$ be an innermost disk on $E$ bounded by a component of $E \cap \boundary
V_C$. Then $\boundary E'$ is an essential curve on the boundary of some component of
$V_C$. If $\mathrm{int}(E')$ is in the exterior of $V_C$ then $\boundary V_C$ is compressible in $\cM$,
and this is impossible as $\cM$ is hyperbolic. Thus $E'$ is a meridian of $V_C$, which
implies it is a component of $\mathcal Z$ with $|\deg(E')| = 1$. 

Thus we have isotoped $S_0$ to reduce intersections with $\boundary V_C$ while not increasing its degree.
We may therefore assume that each component of $S_0 \cap V_C$ is incompressible in
$V_C$. This makes each component either a boundary parallel annulus (degree $0$) or a
meridian disk (degree $\pm1$). Pushing the annuli out of $V_C$ does not affect the
degree and completes the proof. 
\end{proof}
\end{proof}

\section{Fibered fillings of manifolds}
\label{sec:fibered_filling}

In this section we prove Theorem \ref{th:structure} and Theorem \ref{th:main_wp}.  After
setting up notation for Dehn filling and bringing in the branched surfaces from Section
\ref{sec:branched}, we begin the proof in earnest in Section \ref{S:area arguments}. 

\subsection{Fillings and complexity}
\label{S:fill complexity}

Let $\cM$ be a compact $3$-manifold whose boundary is a union
of tori $\partial \cM = \partial_1 \cM \sqcup \cdots \sqcup \partial_r
\cM$, such that $\mathrm{int}(\cM)$ is a complete finite-volume hyperbolic manifold.
A {\em Dehn surgery coefficient} $\beta_i$ on $\partial_i \cM$ is either the
isotopy class of an essential simple closed curve in $\partial_i \cM$,
or ``$\infty$''. 
Each  simple closed curve $\beta_i$ in $\partial_i \cM$  determines a Dehn
filling attaching a solid torus whose meridian is identified with
$\beta_i$, and $\infty$ corresponds to no filling at all. 
Given $\beta = (\beta_1, \ldots \beta_r)$, we denote the manifold
obtained by these specified fillings by $M_\beta$.

Letting $\Delta_i$ denote the set of Dehn surgery coefficients for
$\partial_i\cM$, we say that a property $\mathcal{P}$ holds \emph{for
  all sufficiently long fillings} if there are finite sets $K_i\subset
\Delta_i$ that that $M_\beta$
has $\mathcal{P}$ for all $\beta \in \prod_{i=1}^r(\Delta_i\ssm K_i)$. 
For example, Thurston's hyperbolic Dehn surgery theorem \cite[Theorem
  5.8.2]{Th} states that for any $\ell >0$, the interior of $M_\beta$
is hyperbolic and the lengths of the cores of the filled solid tori
are less than $\ell$, for all sufficiently long fillings.

Now let $\cM$ be as in the statement of Theorem \ref{th:structure} and let $\SS$ be the
collection of all fibers of all fibered fillings of $\cM$.
(Formally $\SS$ is a set of pairs $(S,\beta)$ where $\beta$ determines
a filled manifold $M_\beta$ in which $S$ is a fiber of a fibration.)
Note that we do not require each boundary component of $\cM$ to be filled, so some surfaces in $\SS$ may have boundary.
We also fix a
triangulation ${\bf t}$ of $\cM$.

Given $(S,\beta)\in\SS$, 
isotope $S$ in $M_\beta$
to intersect the added solid tori in meridian disks, and choose it so
the number of disks is minimal. Now let $\cS = S\cap \cM$, and
assume by further isotopy if necessary, that $\cS$ is also transverse to ${\bf t}^{(1)}$ and intersects
it minimally.
Let $\mathring{\SS}=\{\cS:(S,\beta)\in\SS\}$ denote the resulting set of properly embedded surfaces
in $\cM$. Said differently, for each $(S,\beta) \in \SS$, $\cS$ is a
properly embedded surface in $\cM$ which after capping off with a disk
is isotopic in $M_\beta$ to $S$ and, among all such surfaces, minimize the
{\em complexity} defined by the pair $(|\partial \cS|, w(\cS))$ in the lexicographic order, where $w(\cS)$ is the weight of $\cS$.

\begin{lemma}
Each $\cS \in \mathring\SS$ is incompressible and boundary incompressible in $\cM$.
\end{lemma}

\begin{proof}
The proof is standard, but we sketch it for completeness.

Suppose that $\cS$ is compressible and let $D$ be a compressing disk for $\cS$ in $\cM$. Since $S$ is incompressible in its filling $M_\beta$, this disk shares a boundary with a disk $D'$ in $S$, which must meet the cores of the filling tori. Swapping out $D'$ for $D$, which can be done by an isotopy in $M_\beta$, results in a copy of $S$ meeting the cores of the solid tori fewer times, a contradiction. Hence, $\cS$ is incompressible.

Now if $\cS$ is boundary-compressible let $E$ be a boundary compression, so that
$\boundary E$ is the union of an arc $b$ in $\cS$ and an arc $a$ in $\boundary
\cM$. Then $a$ is contained in an annular component $A$ of $\boundary \cM \ssm
\boundary \cS$, and it must be essential there since otherwise $E$ is not an essential
boundary compression. Now $A$ cut along $a$ and attached to two parallel copies of
$E$ gives us a compressing disk for $\cS$, which therefore is parallel to a disk of $\cS$
by incompressibility. But this disk, attached to itself along $b$, forms an annulus
which must be all of $\cS$, so that $\cS$ is a boundary-parallel annulus. This is a contradiction.
\end{proof}

Now Proposition \ref{prop:branched finiteness} gives us a finite
collection $\B_1\ldots,\B_n$ of incompressible branched surfaces in
$\cM$ so that each $\cS\in \mathring\SS$ is fully carried in a
weight-minimizing way in one of them.

For each $\B_i$ let $\SS_{\B_i}$ contain those $(S,\beta)\in\SS$ for which $\cS$ is fully
carried by  $\B_i$. After restricting to ``sufficiently long'' fillings in this
collection, we can make some additional assumptions about $\B_i$:

Let $T_j$ be a component of $\boundary \cM$. If, for $(S,\beta)\in \SS_{\B_i}$, the
coordinate $\beta_j$ only takes on finitely many values, we can exclude all of the
non-$\infty$ values in our definition of sufficiently long. Then unless $\infty$ is also
one of the values, we can ignore $\B_i$ altogether. 

Thus we can assume that for each boundary component $T_j$ of $\cM$ meeting $\B_i$, 
either the fibers of $\SS_{\B_i}$ determine infinitely many slopes in $T_j$ or $T_j$ is not
filled in any of the fillings corresponding to $\SS_{\B_i}$. 
In particular, this means that
each boundary torus $T_j$ of $\partial \cM$ is either disjoint
from $\B_i$,
meets it in a train track and the
complements of these tracks are bigons, or meets $\B_i$ in a disjoint union of parallel simple closed curves in $T_i$ whose complement is a union of annuli.  In the last case $\beta_j=\infty$ for all
$(S,\beta)\in\SS_{\B_i}$ with $\beta$ sufficiently long.

\subsection{Branched surface decomposition}
\label{S:branched surface decomp}
From here on, we work with a single $\B = \B_i$, restricted as above.
The notion of ``sufficiently long filling" will
depend on the particular branched surface $\B$, but finiteness of the
number of branched surfaces provides a uniform notion of sufficiently
long filling, independent of $\B$.
We continue to write $\N = \N(\B)$ for the fibered neighborhood of $\B$ in $\cM$.

Divide $\partial \cM$ into the tori that meet $\B$ -- called
$\partial^B\cM$ -- and the rest, $\partial^F\cM$ ($F$ for
``floating''). 

Given $(S,\beta)\in \SS_\B$, 
let $V_\beta$ be the union of solid tori associated to the
fillings, $M_\beta= \cM \union V_\beta$.  We write $V_\beta^B$ and $V_\beta^F$ for the solid
tori that meet and do not meet $\B$, respectively.
We remark that $\partial V_\beta^F = \partial^F\cM$ and $\partial V_\beta^B \subset \partial^B\cM$ with this containing being proper if there are unfilled boundary components of $\cM$ for $\beta$.

Now let $X = \cM \ssm \N(\B)$ and set  $X_\beta = X \union V_\beta^F$. 
We divide $\partial X_\beta$ into $\partial X_\beta = \partial_h X_\beta \cup \partial_v X_\beta$, where $\partial_h X_\beta = \partial_h \N$, and $\partial_v X_\beta$ consists of annuli which are either annulus components of $\partial_v \N$, annulus components of $X \cap \partial \cM$, or unions of rectangle components of $\partial_v \N$ with bigon regions in $\partial \cM \ssm \N$.  We similarly decompose $\partial X = \partial_hX \cup \partial_v X$, where $\partial_h X = \partial_h X_\beta$, and $\partial_v X = \partial_v X_\beta \cup \partial^F \cM$.

  The foliation by interval fibers of $\N(\B)$ extends first to a
  foliation on each of the non-floating boundary tori $\partial^B\cM$; this is because the
  complement of $\boundary_v\N$ in $\boundary^B\cM$ is a union of bigons and annuli.
  For each $\beta$, the meridians are transversal to this foliation so it
  extends to a foliation of the solid tori $V^B_\beta$, which is
  transverse to (a fixed family of)  meridian disks.

  Call the resulting foliation $\I$, and note that it is defined in $M_\beta \ssm \inter(X_\beta) = \N \cup \partial^B\cM \cup V^B_\beta$.

\begin{figure}[htbp]
\begin{center}
\includegraphics[ width = 1 \textwidth ]{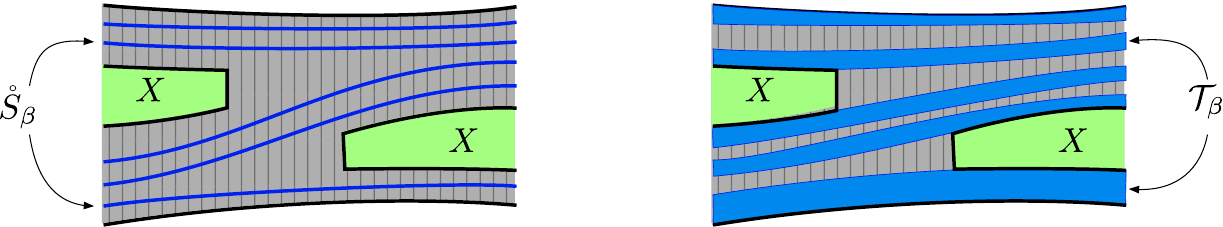}
\caption{Left: $\cS_\beta$ is carried by $\N$. Right: $S_\beta$ is thickened to the
  $I$-bundle $\T_\beta$ (blue). The $I$-foliation is shown in the complementary part of
  $\N$, which is part of $\FF_\beta$.}
\label{fig:schematic}
\end{center}
\end{figure}

\begin{figure}[htbp]
\begin{center}
\includegraphics[ width = .6 \textwidth ]{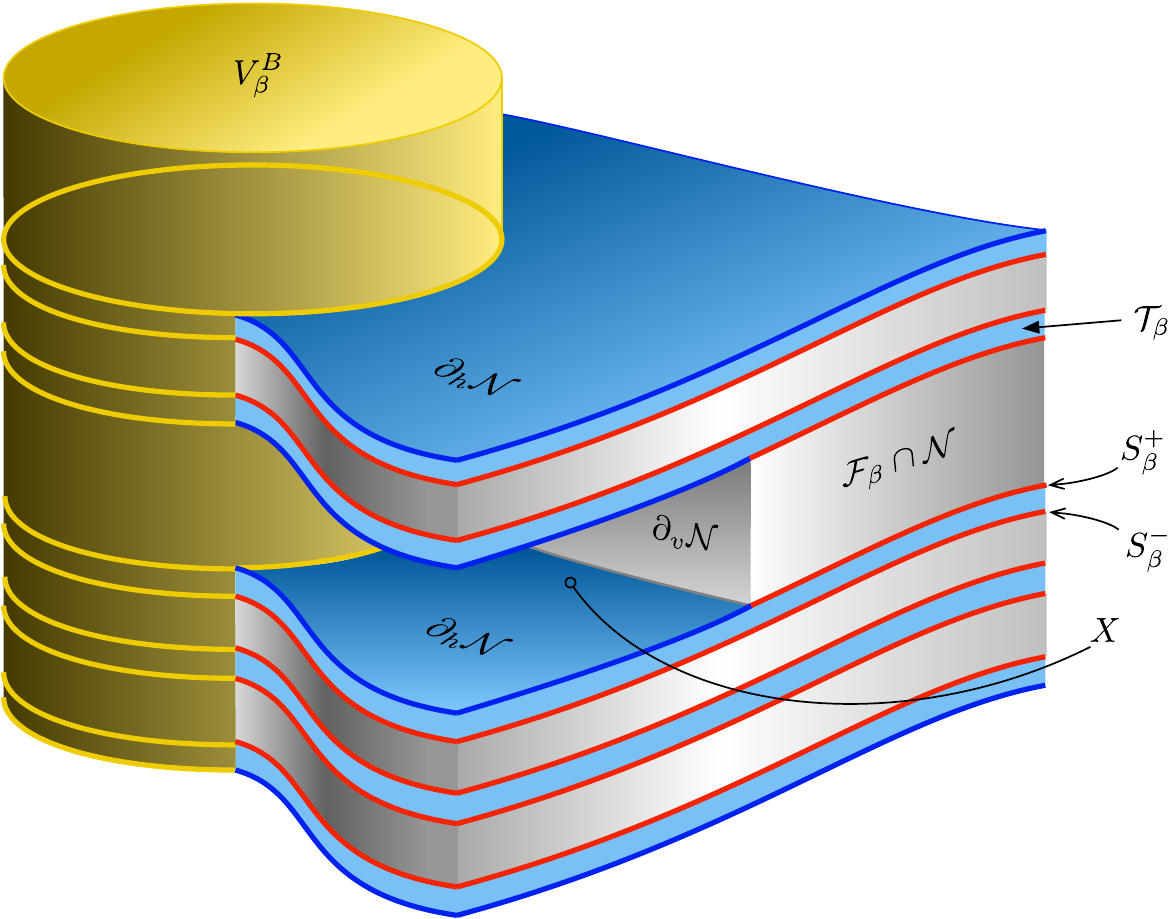}
\caption{A cutaway view of $\N$ near a component of $V^B_\beta$. The $I$-bundle $\T_\beta$
  is indicated in blue inside $\N$ and the complementary $I$-bundle $\FF_\beta$ is in
  gray. The surfaces $S^\pm_\beta$ are indicated with blue edges in $\boundary_h\N$, with
  yellow along the meridians, and with red edges in the remainder $W_\beta$.}
\label{fig:B3d}
\end{center}
\end{figure}

\subsection{Product regions in $M_\beta$}
\label{S:product regions}

Fix $(S,\beta)\in \SS_\B$, and for convenience denote $S$ by $S_\beta$.
We can realize $S_\beta$  as a surface contained in $\N(\B)$ union the solid
tori $V_\beta^B$, in which it is transverse to the foliation $\I$. 
Thicken
$S_\beta$ to make a product $I$-bundle $\T_\beta$ with $I$--fibers being arcs of leaves of $\I$, then isotope $\T_\beta$ 
so that
its boundary contains $\partial_h\N(\B)$. 
This is done by pushing the
boundary surfaces outward along $\I$-leaves until they touch the
endpoints. See Figure \ref{fig:schematic} for a schematic of this, and Figure
\ref{fig:B3d} for a 3-D view near a component of $V_\beta$.

Let $\FF_\beta$ be the closure of
$M_\beta \ssm \T_\beta$. Since $S_\beta$ is a fiber of $M_\beta$, $\FF_\beta$ is also a product
$I$-bundle. We use the $I$-bundle structure to define $\boundary_v$ and $\boundary_h$ for
both $\FF_\beta$ and $\T_\beta$, writing
\[ \partial_h \FF_\beta = \partial_h \T_\beta= S_\beta^+\sqcup S_\beta^-,\]
and noting that $\partial_v \FF_\beta \cup \partial_v\T_\beta = \partial M_\beta$.

Note that $\FF_\beta$ contains $X_\beta$, 
$\FF_\beta\ssm X_\beta$ is contained in $\N \union V^B_\beta$,  and the foliation
$\I$ restricts to a foliation $\I|(\FF_\beta \ssm X_\beta)$.
Hence, $\FF_\beta$ decomposes into the `$I$-foliated part' $\I|(\FF_\beta \ssm X_\beta)$
and the `bounded part' $X_\beta$, as anticipated by our discussion in Section \ref{sec:outlines}.

Our main goal now is to show that $\I|(\FF_\beta \ssm X_\beta)$ is 
(up to isotopy) also the restriction of the product interval foliation on $\FF_\beta \cong
S_\beta\times[0,1]$.  For a summary of our strategy, see the outline in Section \ref{sec:outlines}.
From this we will deduce that $\I$ is the foliation by flow lines of 
the suspension flow for the monodromy of the fiber $S_\beta$ of the given fibration of
$M_\beta$. This is completed in Proposition \ref{Align foliations}.

\subsection{Regions in $\boundary_h\FF_\beta$ and hyperbolic geometry}
\label{S:area arguments}

Let $W_\beta$ denote the closure of $\partial_h\FF_\beta \ssm \partial_h X$. Note that
$W_\beta$ is the union 
of the meridian disks $\D_\beta = \partial_h\FF_\beta \intersect V_\beta^B$ and regions that are
contained in the interior of $\N(\B)$. 

Note that  $\boundary_h X$ is a subsurface with corners in $\boundary_h \FF_\beta
\intersect \cM$ in the sense of Section \ref{sec:wald} -- it is bounded by circles and arcs whose endpoints are on the
circles of $\boundary_h \FF_\beta \intersect \boundary\cM$ -- some of which are boundaries
of the meridian disks and some can be in $\boundary M_\beta$ itself, when
that is nonempty.
See Figures \ref{Fig:boundaryF} and \ref{fig:Econtact} for some example local configurations. 

\begin{figure}[htbp]
\begin{center}
\includegraphics[width = .9 \textwidth]{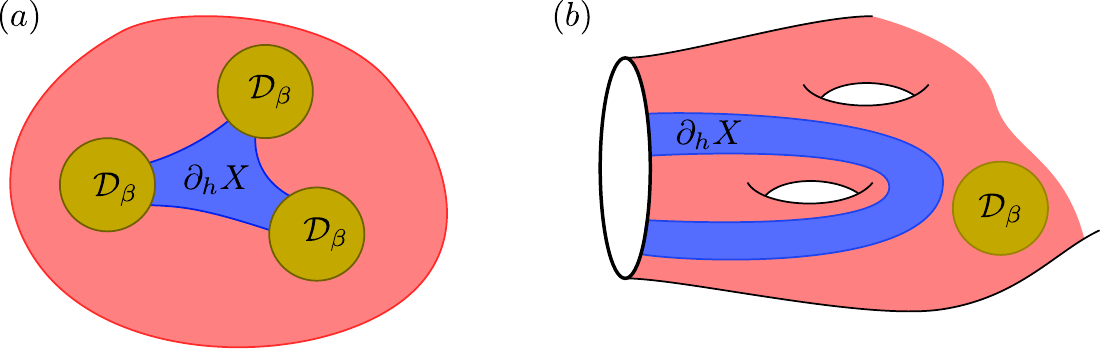}
\caption{
Examples of the decomposition of $\boundary_h\FF_\beta$. $W_\beta$ is the union of red
and the yellow meridians. 
(a) shows a component of $W_\beta$ of low complexity containing meridians, which is
ruled out for long fillings by Lemma \ref{meridians in W}. (b) shows part of a larger component of
$W_\beta$ adjacent to the boundary of $M_\beta$.}
\label{Fig:boundaryF}
\end{center}
\end{figure}

For any component $E$ of $W_\beta$, let $\hat E$ denote the union of $E$ with any disks of
$\boundary_h\FF_\beta \ssm \inter(E)$ that meet $\partial(\boundary_h\FF_\beta)$ in at most one arc. Note
that $E$ and $\hat E$ are  subsurfaces with corners of $\boundary_h\FF_\beta$ in the sense of Section \ref{sec:wald}.

Lemmas \ref{meridians in W} and \ref{no trivial components} restrict the structure of
components of $W_\beta$ for sufficiently long fillings.

\begin{lemma}\label{meridians in W}
Fix an integer $k$. For all sufficiently long fillings, if $E$ is a component of $W_\beta$ with $\hat\chi(\hat E)=k$, then $\hat E$ is disjoint from $\D_\beta$.
\end{lemma}
Here $\hat\chi$ is modified Euler characteristic as in Section \ref{sec:wald}.

\begin{proof} 
  Consider first those $\beta$ for which $\boundary M_\beta = \emptyset$, in which case $\boundary_h\FF_\beta$ is a
  closed surface and $\hat\chi(\hat E) = \chi(\hat E)$.

  Identify $\cM$ once and for all with the complement of some standard cusps in the
  finite-volume hyperbolic metric on $\inter(\cM)$. Thurston's Dehn filling theorem tells us that $\cM$ embeds nearly
isometrically in the hyperbolic structure on $M_\beta$, for sufficiently long fillings, so that its
complement $\inter(V_\beta)$ consists of the Margulis tubes for the
corresponding curves. Moreover the radii of these tubes are
arbitrarily large for sufficiently long fillings $\beta$. 
Note that $\boundary\hat E$ is contained in $\boundary \boundary_h X$, which is independent of
$\beta$. There is therefore some bound $\kappa$, independent of $\beta$,  on the total curvature of
$\boundary \hat E$ in $M_\beta$. 

Fix a triangulation of $\hat E$ with vertices on the boundary so that each triangle has at
most one edge on the boundary. For each $\beta$ let $f_\beta \colon \hat E
\to M_\beta$ be a map which is a ruled surface on each triangle
and is homotopic rel
boundary to the inclusion map. Then the
Gauss-Bonnet theorem for the induced metric on $f_\beta(\hat E)$ gives us
$$
\Area(f_\beta(\hat E)) \le -2\pi\chi(\hat E) + \kappa.
$$
Note that the right hand side is independent of $\beta$.

The following lemma (whose proof appears below) now allows us to finish the proof:
  \begin{lemma}\label{small surface homotopy}
    Given $A\ge 0$ there is an $R\ge 0$ such that the following holds.  Let $V$
    be a hyperbolic Margulis tube of radius $R$ and let $W$ be a compact, connected, oriented
    surface with a map $f\colon(W,\boundary W) \to (V,\boundary V)$ such
    that $\Area(f(W)) < A$. Then $f$ is homotopic rel $\boundary W$ into
    $\boundary V$. 
  \end{lemma}

  Applying this to each intersection of $f_\beta(\hat E)$ with the Margulis tubes $V_\beta$, 
  and choosing $\beta$ long enough to give the needed value for the tube radii,
we find that
we can homotope $\boundary_h\FF_\beta$ in $M_\beta$ (and hence
$S_\beta$ in $M_\beta$) to remove its intersections with the
cores of $V_\beta$ that occur in $\hat E$. 
Since $S_\beta$ was already chosen to minimize these intersections in its isotopy class,
and Proposition \ref{prop:FHS} says that it must also minimize them in its homotopy
class, we conclude that $\hat E$ could not in fact have contained any meridian disks. 

Now consider those $\beta$ for which some specific set of
coordinates is $\infty$. The corresponding tori are unfilled so
$\boundary M_\beta$ is nonempty, and each component is associated to a
cusp in $M_\beta$. We adapt the argument to handle these cusps.

We identify $M_\beta^\partial \equiv \cM \ssm \partial M_\beta$ with the complement of the remaining standard cusps
in the finite-volume hyperbolic structure on $\inter(\cM)$. 
Again for sufficiently long fillings we can embed this nearly
isometrically in the hyperbolic structure on $\inter(M_\beta)$ as the
complement of the Margulis tubes of the filled boundary components. 

Let $\hat E_\beta = \hat E\cap M^\boundary_\beta$, which is $\hat E$
minus the arcs and curves of its boundary that lie in $\boundary
M_\beta$. We can (after suitable isotopy of the hyperbolic metric) assume that the boundary
arcs of $\hat E_\beta$ are, in some neighborhood of the boundary,
totally geodesic rays exiting the cusps. Thus the ends of $\hat
E_\beta$ can be deformed to finite-area cusps or ``spikes'' (regions
between two asymptotic geodesics).

Now our triangulation of $\hat E_\beta$ can be chosen with an ideal
vertex at each end of the surface, and when we homotope it rel
boundary (and rel ideal boundary points) to a ruled surface, the 
Gauss-Bonnet theorem applies again, but with an additional $\pi$ in the
boundary term for each spike. Thus we have
$$
\Area(f_\beta(\hat E_\beta)) \le -2\pi\hat\chi(\hat E) + \kappa.
$$
Lemma \ref{small surface homotopy} again applies, allowing us for
sufficiently long fillings to find
a proper homotopy of $\hat E_\beta$ rel boundary and rel ends which
removes all intersections of $\hat E$ with $V_\beta$. A standard
argument in the collar of $\boundary M_\beta$ allows us to obtain a
homotopy of $\hat E$ itself which does the same thing. Again we
conclude that $\hat E$ could not have contained any meridian disks. 
\qedhere

\end{proof}

We now supply the proof of Lemma \ref{small surface homotopy}.
\begin{proof}
We can write $V = D^2 \times S^1$, where $D^2\times\{t\}$ are totally geodesic meridian
disks for $t\in S^1$, and let $p:V\to D^2$ be the projection, in such a way that $p$ is
area preserving on the meridian disks. 
    
We can write the area form of the meridian disks of $V$ 
explicitly in cylindrical coordinates $(r,\theta,z)$ in the universal
cover of $V$, namely $\alpha = \sinh r\, dr\wedge d\theta$. Note that
this is closed, and evaluates to $2\pi(\cosh R-1)$ on a meridian disk.
Thus $\alpha$ represents $2\pi(\cosh R -1)p^*(\eta)$ where $\eta$ is the fundamental class
of $H^2(D^2,\boundary D^2)$. 

Now the map $(p \circ f)_* : H_2(W,\boundary W) \to
H_2(D^2,\boundary D^2)$ is just multiplication by an integer, the degree $\deg(p \circ
f)$. We can therefore compute this degree
by integrating $\alpha/2\pi(\cosh R-1)$. That is,
$$
\deg (p \circ f) = {1\over 2\pi(\cosh R-1)}\int_W f^*\alpha.
$$
On the other hand, 
because (fiberwise) orthogonal projection in the tangent bundle of $V$ to the meridian
disk direction is 
contracting, we also have 
$$
\left|\int_Wf^*\alpha\right| \le \Area(f(W)) < A.
$$
Thus we have
$$ |\deg(p \circ f)| < A/2\pi(\cosh R-1).$$
Since the degree is an integer, for suitably large $R$ we conclude $\deg(p \circ f) = 0$.

  Now a relative version of the Hopf Degree Theorem (see for example the
   Extension Theorem in \cite[Chapter 3]{guillemin2010differential}) tells us that,
   since $\deg(p \circ f)=0$, there is a homotopy $G$ rel $\boundary W$ taking $p \circ f$
   to a map $g$ that takes values in $\boundary D^2$. Writing $f = (f_1,f_2)$ where $f_1 = p \circ f$ and applying the homotopy $G$ to the first coordinate completes the proof.
 \end{proof}

Recall from Section~\ref{sec:wald} that a component of $W_\beta$ is called \emph{trivial} if it is contained in a disk $D$ which is either contained in the interior of $S_\beta^\pm$ or meets $\partial S_\beta^\pm$ in a single arc. Note that if a component $E$ of $W_\beta$ is not trivial, then $\partial \hat E \ssm \partial S_\beta^\pm$ consists of homotopically essential curves and essential proper arcs in $S_\beta^\pm$.

 \begin{lemma}\label{no trivial components}
   For sufficiently long fillings,    $W_\beta$ has no trivial
   components.
  \end{lemma}

\begin{figure}[htbp]
\begin{center}
\includegraphics[width = .7 \textwidth]{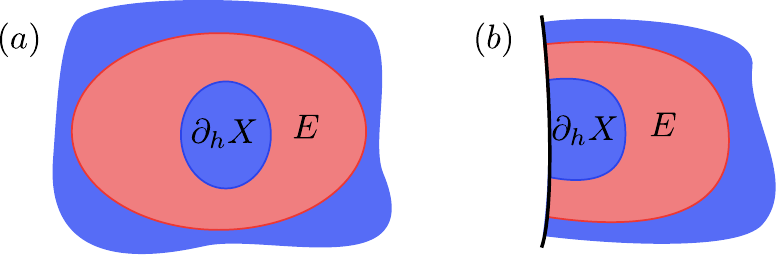}
\caption{Two trivial components of $W_\beta$ which result in a disk and half-disk of contact.}
\label{fig:Econtact}
\end{center}
\end{figure}

\begin{proof}
First suppose that a component $E$ of $W_\beta$ were contained in a disk that does not meet $\boundary( \boundary_h\FF_\beta)$.
Then $E$ would be a (possibly punctured) disk,
so $\hat E \subset \boundary_h\FF_\beta$ would also be a disk not meeting $\partial( \boundary_h\FF_\beta)$.

Lemma \ref{meridians in W} 
implies that for sufficiently long fillings, $\hat E$ (and hence $E$) contains no meridians. 
In particular $\partial \hat E$ must be a smooth boundary component of
$\partial_h X$ (as in Figure \ref{fig:Econtact}(a)). That is, $\hat E \subset \N$ and $\partial \hat E \subset \partial_v \N$. 
But then, as in \cite[Claim 1]{floyd1984incompressible}, we can isotope $\hat E$ to slide $\partial \hat E$ 
into $\interior (\partial_v \N)$ thereby producing a disk of contact for $\B$. See Figure \ref{fig:contact}.
\begin{figure}[htbp]
\begin{center}
\includegraphics[width = .5 \textwidth]{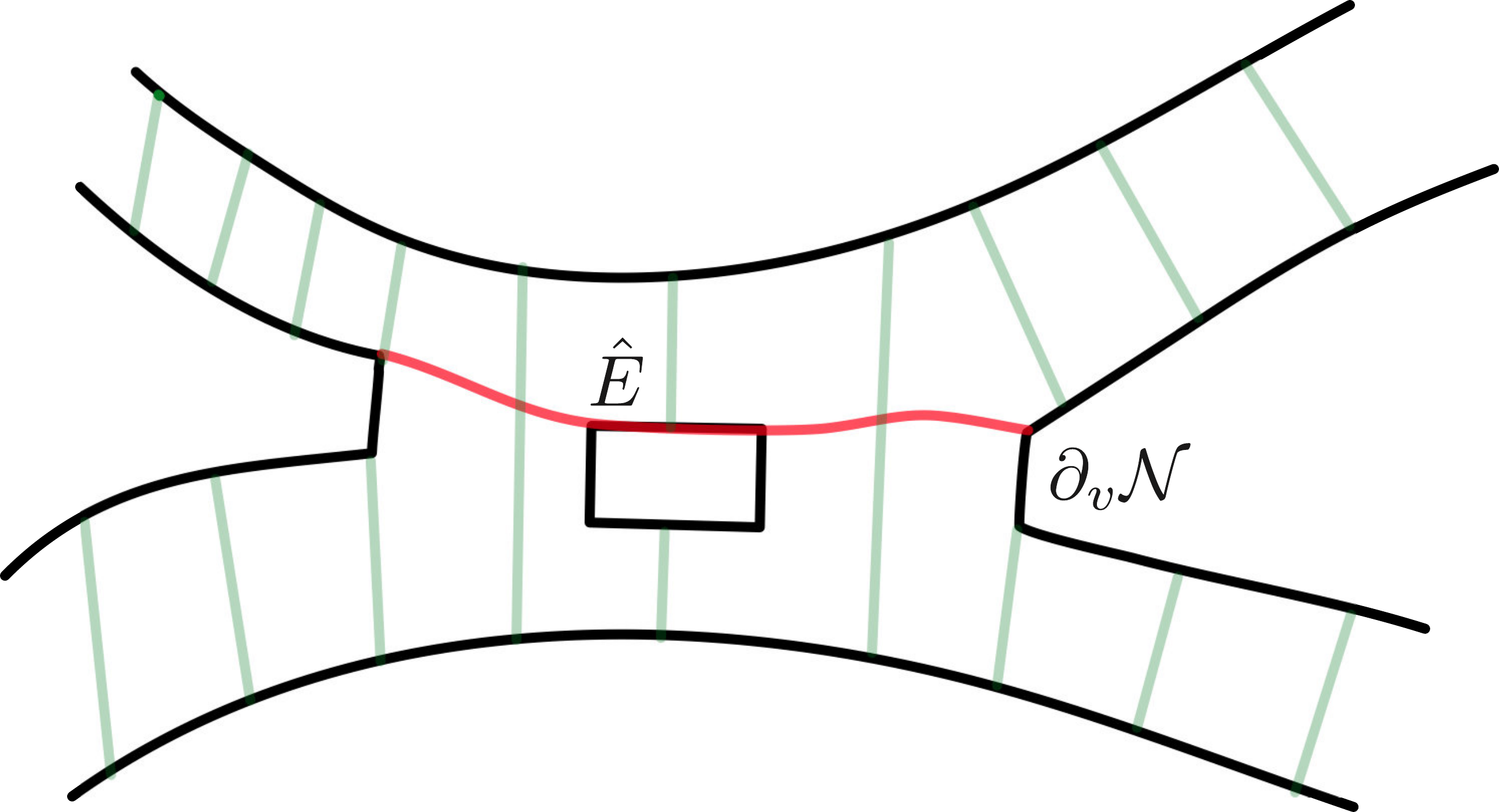}
\caption{A disk $\hat E$ (red) in $\N$ with boundary in $\partial_v \N$. Sliding $\partial \hat E$ slightly downward produces the disk of contact. }
\label{fig:contact}
\end{center}
\end{figure}
This is a contradiction and the proof is complete in this case.

If instead $E$ were contained in some disk in $\boundary_h\FF_\beta$ which meets $\partial
(\boundary_h\FF_\beta)$ in a single arc, then $\hat E$ would also be a disk whose boundary
has a single arc in common with $\partial (\boundary_h\FF_\beta)$ (as in Figure
\ref{fig:Econtact}(b)). Just as in the previous case, Lemma \ref{meridians in W} implies that $\hat E$ contains no meridians and so is contained in $\N$.
This time since $\partial E$ consists of one arc in $\partial \cM$ and one arc in $\partial_v \N$
we see that $\hat E$ produces a half-disk of contact for $\B$ and also results in a contradiction.
\end{proof}

\subsection{Transverse orientability of $\B$}

\begin{lemma}\label{transverse orientation}  
The branched surface $\B$ is transversely orientable.  Equivalently, the foliation $\I$, defined on $M_\beta \setminus X_\beta = \N(\B) \cup V^B_\beta$, is orientable.
\end{lemma}

Transverse orientability of $\B$ is clearly equivalent to the orientability of the foliation of $\N = \N(\B)$ by $I$--fibers.  Any orientation on this foliation of $\N$ easily extends to an orientation of the foliation $\I$.  

The proof is an adaptation of an argument of Oertel \cite{oertel1986homology} who proved
that branched surfaces constructed from Thurston-norm minimizing surfaces are transversely orientable.
In our case we appeal to the notion of complexity defined in Section~\ref{S:fill
  complexity} and minimization in an isotopy class, rather than Thurston's complexity,
$\chi_-$, minimized over a homology class.

\begin{proof}
  Suppose $\B$ is not transversely orientable. Let $(S_\beta,\beta)\in \SS_\B$, and fix a transverse orientation on
  $S_\beta$, and hence on $\cS_\beta$. We will show that, for sufficiently long $\beta$,
  this leads to a contradiction.
  
 Since $\cS_\beta$ is fully carried on $\B$ there must be a
  branch of $\B$ where the orientations are inconsistent. So there is
  a region in $\N$ where there are two adjacent sheets of $\boundary_h\FF_\beta$
  whose transverse orientations point into (or out of) the region of $\FF_\beta$
  between them. Extending this region maximally along 
  $\FF_\beta\ssm\inter(X_\beta)$, one obtains a subset of $\FF_\beta$ of the form $P \cong E\times [0,1]$ where $E\times\{0\}$ and $E\times\{1\}$ are identified
 with components $E_0$ and $E_1$ of $W_\beta$ on the {\em same component} of 
 $\partial_h \FF_\beta$, which we denote $S_\beta^+$ without loss of generality.

 Just as at the beginning of the proof of Lemma \ref{sub products}, if $\hat E_0 \neq E_0$
 then we can extend the product region $P$ to $\hat P \cong \hat E \times  [0,1]$ such
 that $  \hat E \times \{0\}$ and $\hat E \times \{1\}$ are identified with $\hat E_0$ and
 $\hat  E_1$. 
In a bit more detail, any disk $D_0$ of $\hat E_0 \ssm E_0$ corresponds to a disk $D_1$ in $\hat E_1 \ssm E_1$ (by incompressibility and boundary incompressibility of $S_\beta^+$ in $\FF_\beta$) and these disks, along with a foliated annulus of $P$, cobound a ball in $\FF_\beta$ (by irreducibility of $\FF_\beta$). Each such ball can be foliated with intervals, extending the foliation of $P$, and $\hat P$ is the subset of $\FF_\beta$ obtained by taking the union of $P$ with all such foliated balls.
 
Since $(S^+_\beta,\boundary S^+_\beta) \hookrightarrow (\FF_\beta,\boundary_v\FF_\beta)$ is a
homotopy equivalence of pairs, 
the homotopy of $\hat E_0$ to $\hat E_1$ along $\hat P$ implies that these two
regions are homotopic in $S_\beta^+$ through maps preserving $\partial S^+_\beta$.
But on the other
hand, the regions $\hat E_0$ and $\hat E_1$ are disjoint. This
is only possible if $\hat E$ is a disk meeting $\partial S^+_\beta$ in at most two arcs
or an annulus not meeting $\partial S^+_\beta$. 

Lemma~\ref{no trivial components} rules out disks meeting $\partial S_\beta$ in at most one arc
for sufficiently long fillings.  Since orientability of $\B$ is independent of filling
slope $\beta$, we may assume that $\hat E$ does not have this form.  Therefore, $\hat E$ is
either a disk meeting $\partial S_\beta$ in exactly two arcs (a ``rectangle'') or is an essential
annulus. In either case, Lemma \ref{meridians in W} tells us that for sufficiently long
fillings, $\hat E_0$ and $\hat E_1$ contain no meridians. 

Let us first consider the annular case.
Thus, we have that $\hat E_0$ and $\hat E_1$ are two annuli  bounded by smooth curves and
parallel in $S^+_\beta$, and $\hat P$ is a solid torus. 
The vertical boundary  $\boundary_v\hat P$ consists of two annuli identified with $\boundary\hat E \times
[0,1]$. Each of them is incompressible with boundary on $S^+_\beta$ and must therefore be boundary
compressible. By irreducibility of $\FF_\beta$, they must each cobound a solid torus with
an annulus in $S^+_\beta$. Choosing $U$ the innermost of these two solid tori, we see that
$U$ meets $S^+_\beta$ in an annulus $A$ between $\hat E_0$ and $\hat E_1$, and meets $\hat
P$ in one of the annuli of $\boundary_v\hat P$, which we denote $A_v$. See Figure
\ref{monogon2}. The meridian disk $m_\beta$ of $U$, given by the boundary compression of $A_v$,
meets $A_v$ in a single arc.

Though $U$ may contain other foliated solid tori parallel to $\hat P$, we can take $\hat P$
innermost so that $U$ is a component of $X_\beta$.
Note that $U$ cannot be a solid torus in $\cM$, since that would make $m_\beta$ a monogon for
$\B$. Thus it must contain a solid torus $V$ of $V_\beta$.

 \begin{figure}[htbp]
\begin{center}
\includegraphics[width = .5 \textwidth]{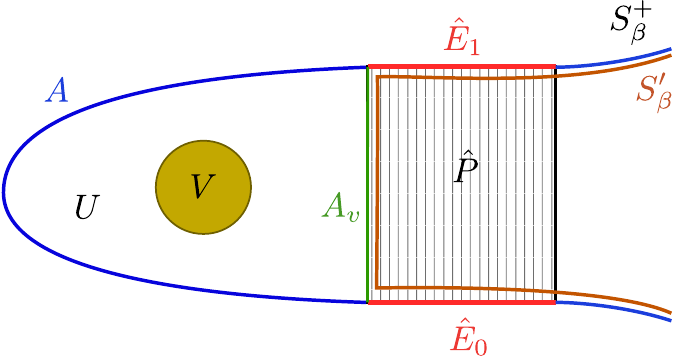}
\caption{This figure crossed with $S^1$ illustrates the annulus case in Lemma
  \ref{transverse orientation}. Crossed with $[0,1]$ and without the yellow disk, it
  illustrates the rectangle case.}
\label{monogon2}
\end{center}
\end{figure}

   Now replace $A$ with $A_v$ in $S^+_\beta$ to produce:
   \[ S_\beta' = (S_\beta^+ \setminus A )\cup A_v. \]
   Push $S_\beta'$ slightly further into $\N(\B)$ so that it misses $A_v$.

  This surface is isotopic to $S_\beta$ (through $U$), because the meridian $m_\beta$ meets $A_v$ and $A$ in a
  single arc. The isotopy does not change the intersection with the cores of
  the solid tori, but reduces the intersection with the 1-skeleton of
  ${\bf t}$. This is because, by Proposition \ref{prop:branched finiteness}, the annulus
  $A_v$, which is a component of $\boundary_v\N(\B)$, contains arcs of ${\bf t}^{(1)}$
  which intersected $S^+_\beta$ but miss $S'_\beta$. 
   This contradicts the complexity-minimizing choice of
   $S_\beta$, also made possible by Proposition \ref{prop:branched finiteness}.

It remains to consider the case where each $\hat E_i$ is an essential rectangle,
i.e. a nontrivial disk in $S_\beta^+$ which meets $\partial S_\beta^+$ along two arcs.
Figure \ref{monogon2} again describes the situation, but one should interpret the diagram
cross $[0,1]$ instead of $S^1$. Thus we see $\hat E_1$ and $\hat E_2$ as rectangles in
$S^+_\beta$, and $\hat P$ as a cube, with its front and back faces (the diagram rectangle
crossed with $\{0\}$ and $\{1\}$) lying in $\boundary M_\beta$. 
The arcs labeled $A_v$ and $A$ represent rectangles lying in $\boundary_v \N$ and
$S^+_\beta$ respectively, and their union $A_v\union A$ is a properly embedded
annulus, whose boundary circles lie in the toroidal boundary $\boundary M_\beta$. 

Suppose it is not null-homotopic. Then, because $\cM$ is hyperbolic 
$A_v\union A$  must be a boundary-parallel annulus. This means that each arc connecting its
boundaries can be deformed rel endpoints to $\boundary M_\beta$. Applying this to such an
arc lying just in $A$, we obtain a boundary-compression of $S^+_\beta$, which is a
contradiction. 

Thus $A_v\union A$ is null-homotopic so its boundary circles bound disks in
$\boundary M_\beta$ because $\boundary M_\beta$ is incompressible.
By irreducibility then the region between $A_v\union A$ and these two disks is a ball, 
labeled $U$ in the figure.

Again by choosing $\hat P$ innermost we can assume that $U$ is a component of $X_\beta$. 
There is no component of $V_\beta$ in $U$ this time, since it is a ball, so it is in fact
a component of $X$. This means that
a disk $m_\beta$ constructed as a compression of the annulus $A_v\union A$ is an actual
monogon for $\B$ in $\cM$, and this is a contradiction to the incompressibility of $\B$.

This contradiction implies that $\B$ is orientable. 
\end{proof}

\subsection{Product structures are tame}

We can now assemble the proof of the key fact that, for sufficiently long fillings, 
the  foliation $\I|(\FF_\beta \ssm X_\beta)$ comes from a product structure on $\FF_\beta$.

\begin{proposition}\label{Align foliations}
For sufficiently long fillings, the foliation $\I|(\FF_\beta \ssm X_\beta)$ can be extended to a foliation of all of $\FF_\beta$.  Consequently, $\I|(\FF_\beta \ssm X_\beta)$ agrees with a product foliation $\FF_\beta \cong S_\beta \times [0,1]$ (up to isotopy), and each component of $X_\beta$ is a subproduct of the associated product structure on $\FF_\beta$.
\end{proposition}

\begin{proof}
On each component $P$ of $\FF_\beta \ssm X_\beta$, the foliation $\I|(\FF_\beta \ssm
X_\beta)$ by intervals determines a product structure $E \times [0,1]$.  Lemma
\ref{transverse orientation} implies that $\I$ is orientable, and hence any leaf must
intersect both components of $\partial \FF_\beta$.  In particular, $E\times\{0\}$ and $E\times\{1\}$ are components of $W_i$ on {\em opposite} sides of $\partial \FF_\beta$.
Furthermore, Lemma \ref{no trivial components} implies that for all sufficiently long fillings $E\times\{0\}$ is not trivial. 
Therefore, Lemma \ref{sub products} implies that the product structure of $\FF_\beta$ can be isotoped so that it matches the product structure determined by the foliation $\I|(\FF_\beta \ssm X_\beta)$.  
\end{proof}

\subsection{Fixed-fiber reduction and the completion of the proof.} 
\label{S:fiber theorem proofs}
We are now ready for the proofs of Theorems \ref{th:structure} and \ref{th:main_wp}.

\begin{proof}[Proof of Theorem~\ref{th:structure}] As in our setup so far (Section
  \ref{S:branched surface decomp}) we restrict attention to a single branched surface $\B$
  and the associated fillings and fibers $\SS_\B$.
  
For $(S,\beta)\in \SS_\B$ sufficiently long we can apply Proposition~\ref{Align
  foliations}, which tells us that the foliation $\I$ on $\N\cup V^B_\beta$ extends to
$X_\beta$, giving a product foliation on $\FF_\beta$ for which $X_\beta$ is a subproduct.
The product structures on $\FF_\beta$ and $\T_\beta$ define a specific mapping torus
structure on $M_\beta$ and hence suspension flow $(\psi_s)$ on $M_\beta$ (well-defined up
to reparameterization).  By construction of $\I$ on $V_\beta$, we see that $V_\beta$ is
invariant by $(\psi_s)$ and the $I$--fibers of $\N$ are arcs of flow lines. 

In particular the cores of $V^B_\beta$ are already vertical with respect to this
structure. Left to handle are the solid tori $V^F_\beta$, which may still be knotted in
$\FF_\beta$. To address this we need to establish additional uniformity which will allow
us to invoke a theorem of Otal about unknottedness of short geodesics in Kleinian surface groups.

Let $\check M_\beta = \cM \cup V^F_\beta$ be the manifold obtained by filling $\cM$ along only
the floating tori $V^F_\beta$, and note that $\check M_\beta = \N \cup X_\beta = M_\beta
\ssm \inter(V^B_\beta)$.  Since
$(\psi_s)$ preserves $V^B_\beta$, it restricts to a flow on $\check M_\beta$.

Now we observe that for {\em any} surface $\cS$
fully carried by $\B$ -- for example we can start with $(S,\beta_0)$ in $M_{\beta_0}$ and
remove intersections with $V^B_{\beta_0}$ -- we can embed it in $\N$ in the standard way
and view it as a surface in $\check M_\beta$ for a different value of $\beta$. We
henceforth fix such a $\cS$ and allow $\beta$ to vary.
Then $\cS$ is transverse to the flow $(\psi_s)$ and in fact meets every flow line in
forward and backward time;
otherwise a flow half-line $\psi_{[s,\infty)}$ or $\psi_{(-\infty,s]}$ would miss $\N$ meaning that it was trapped in some component of
$X_\beta$, which is impossible since $X_\beta$ is a product and the flow lines intersect
it in compact arcs of the product foliation.  Hence, there is a well-defined first return
map $\phi \colon \cS \to \cS$ of $(\psi_s)$ and $\check M_\beta$ is the mapping torus on
$\phi$.  

As in our previous constructions, we 
thicken $\cS$ to a product $\T_{\cS}$ and push it out along $I$--fibers of $\N$, so that
$\partial_h\T_{\cS}$ contains $\partial_h \N$.  The closure of the complement $\FF_{\beta,\cS}$
is also a product, since $\cS$ is a fiber in $\check M_\beta$.  Since the product
structures on each of $\T_{\cS}$ and $\FF_{\beta,\cS}$ are compatible with $(\psi_s)$, it
follows that $X_\beta$ is a subproduct of $\FF_{\beta,\cS}$. Note that this product
structure is the same up to isotopy as the product structure inherited from $\FF_\beta$,
since both are determined by the decomposition 
$\boundary X_\beta = \boundary_h X_\beta \union \boundary_v X_\beta$. 

Consider the infinite cyclic cover $N_{\beta,\cS}$ of $\check M_\beta$ associated to $\cS$.
Otal's theorem \cite{otal1995nouage} implies that there is an $\ell >0$ depending only on
$|\chi(\cS)|$ so that if the cores of the solid tori $V_\beta^F$ have hyperbolic length
less than $\ell$ then their lifts to $N_{\beta,\cS} \cong \cS \times \R$ are level. 
(We note that although Otal only explicitly treats the case where $\cS$ is closed, the general
case is similar. Alternatively, the version needed here is explicitly stated by Bowditch \cite[Theorem 2.2.1]{bowditch2011ending} 
and also follows directly from a more general result of Souto \cite{souto2008short}.)
Henceforth, we consider only $\beta$ sufficiently long so that the cores of $V_\beta^F$ have
length less than $\ell$, which is again possible by Thurston's Dehn surgery theorem.

The product structure $N_{\beta,\cS} \cong \cS \times \R$ is obtained by gluing the $\BZ$--indexed lifts of $\T_{\cS}$ and $\FF_{\beta,\cS}$ together.  Observe that all $V_\beta^F \subset X_\beta \subset \FF_{\beta,\cS}$, and that $X_\beta$ is a subproduct of $\FF_{\beta,\cS}$ by Proposition~\ref{Align foliations}.
By Lemma \ref{different straightenings}, working in one of the lifts of $\FF_{\beta,\cS}$ to $N_{\beta,\cS}$ (which projects homeomorphically to $\check M_\beta$), the cores of $V_\beta^F$ are level in the product structure of $\FF_{\beta,\cS}$, and further more the isotopy is supported in $X_\beta$.  That is, after an isotopy, we may assume that the cores of $V_\beta^F$ are level with respect to $\cS$.

But the product structures on $X_\beta$ are isotopic, so 
we see that  the cores of $V^F_\beta$ are level with respect to
$\cS_\beta$, and hence $S_\beta$ as well.  Since the cores of $V^B_\beta$ are already transverse, it follows that
the cores of all filling solid tori $V_\beta$ are standard, and since $M$ is obtained from
$M_\beta$ by drilling out these cores, we are done. 
\end{proof}

We conclude this section by proving Theorem \ref{th:main_wp}.
\begin{proof}[Proof of Theorem \ref{th:main_wp}] 
Fix $L >0$. If $\phi \in \Phi_{wp}(L)$, then $\vol(M_\phi) \le \frac{3}{2}\sqrt{2\pi}L$ by Theorem \ref{th:vol}. By the J\o rgensen--Thurston theorem \cite[Theorem 5.12.1]{Thu78}, there is a finite collection $\M_0$ of hyperbolic $3$-manifolds of volume at most $V =  \frac{3}{2}\sqrt{2\pi}L$ so that every hyperbolic $3$-manifold of volume at most $V$ is obtained from some manifold in $\M_0$ by hyperbolic Dehn filling. 

Since the set $\M_0$ is finite, Theorem \ref{th:structure} gives that
after excluding at most finitely many slopes per boundary component
per manifold in $\M_0$, all other fillings of the manifolds in $\M_0$
have cores that are level or transverse. 

For each $\cM\in \M_0$ choose a boundary component and consider the
finitely many manifolds obtained by filling along the excluded slopes
for that boundary. Let $\M_1$ be the collection of all such fillings
over all members of $\M_0$. 
Since $\M_1$ is also finite, we may repeat the process of applying
Theorem \ref{th:structure}. Proceeding inductively, we terminate when
no more fillings are needed to obtain elements of
$\Phi_{wp}(L)$. Along the way we have accounted 
for all of the members of $\Phi_{wp}(L)$, showing that they come from our
combined union of finite families by drilling along level or transverse curves.
\end{proof}

We end this section by describing an example of $1$--manifolds in a sequence of Dehn fillings on a compact manifold with hyperbolic interior, such that (a) each manifold fibers in infinitely many ways and (b) for all sufficiently long fillings, the $1$--manifold in each filling identified by the theorem is level with respect to one fiber and transverse with respect to another.
\begin{example}[Level/transverse is relative] \label{ex:2structures}
Let $a$ and $b$ be homologous nonseparating curves that fill a closed surface $S$ of genus $g \geq 2$, and let $\phi_n = T_a^n \circ T_b^{-n}$ where $T_c$ is the Dehn twist about $c$.
These mapping classes are pseudo-Anosov for $n\ge 1$ by Thurston's construction \cite{Th}, and so the mapping tori $M_n = M_{\phi_n}$ are hyperbolic $3$--manifolds by Thurston's hyperbolization theorem \cite{thurston1986hyperbolic,otal2001hyperbolization}.

We can view $M_n$ as obtained from a sequence of Dehn surgeries on $C = a \times \{\frac23\} \cup b \times \{\frac13\}$ in the mapping torus of the identity $M_0 = S \times [0,1]/(x,1) \sim (x,0) \cong S \times \BS^1$.
By the J\o rgensen--Thurston theorem \cite[Theorem 5.12.1]{Thu78}, the sequence of hyperbolic manifolds $M_{n} = M_{f_n}$ limits to the cusped hyperbolic manifold $\cM =M_n \ssm (a^* \cup b^*)$ obtained by removing the geodesic representatives of $a$ and $b$ from $M_n$ for $n\gg1$.
This is the manifold obtained by drilling along (disjoint copies of) the level curves $a$ and $b$ of the fiber $S$ in any manifold of the sequence, and is homeomorphic to $M_0 \setminus C$.

Now chose a different fiber $S_n$ of the manifold $M_n$ over the same fibered face as $S$ into which $a$ and $b$ {\bf cannot} be homotoped (see \S\ref{S:Thurston-Fried}).
To see that it is possible to find such a fiber $S_n$, first observe that the Poincar\'e dual of $S$ lies in the subspace of $H^1(M)$ which vanishes on the homology class of $a$ (and $b$, since $a$ and $b$ are homologous).  
The linear subspace of $H^1(M)$ consisting of classes that vanish on this homology class has codimension $1$ since $\phi_n$ acts trivially on $H_1(S)$ (that is, $\phi_n$ is in the Torelli group), and hence any element of $H^1(S)$, in particular one that is nonzero on $a$, extends to an element of $H^1(M)$.

Since $a^*$ and $b^*$ become arbitrarily short as $n$ tends to infinity, these curves are necessarily a part of $\C$ from Theorem \ref{th:structure}; in fact, their union is equal to $\C$.  According to that theorem, $a$ and $b$ must be transverse in $S_n$, for $n$ sufficiently large. 
In fact, we see that $\cM$ is a fibered manifold with fibers $S_n$ punctured along $a \cup b$.

Therefore, for all $n$ sufficiently large, $a^* \cup b^* = \C$ in $M_n$ is level with respect to the fiber $S$ but transverse with respect to the fiber $S_n$.
\end{example}

\section{Bounding Weil-Petersson translation length} 
\label{sec:estimates}

In this section we prove the following theorem from the introduction.\\

\noindent {\bf Theorem~\ref{T : twisted Linch bound 2}.}
{\em There exists $c > 0$ so that if $\phi \colon S \to S$ is a pseudo-Anosov on a closed surface, $\alpha \subset S$ is a simple closed curve with $\tau_\alpha = \tau_\alpha(\phi) \geq 9$, and $k \in \mathbb Z$, then
\[ \Vert T_\alpha^k \circ \phi \Vert_{wp} \leq \Vert \phi \Vert_T  \sqrt{c |\chi(S)|}.\]}

The bound on translation length in the theorem is obtained by producing an explicit $(T_\alpha^k \circ \phi)$--invariant path $\mathbb R \to \T(S)$ and bounding the WP-length of a fundamental domain for the action of $\langle T_\alpha^k\circ \phi \rangle$.
The path is constructed as a {\em leaf-wise conformal structure} on the mapping torus of $T_\alpha^k \circ \phi$.  

The remainder of this section is concerned with constructing the required structure and proving the  bound on the Weil-Petersson translation length.  We begin in Section \ref{S:leaf-wise conformal structure} where we make precise what we mean by a suspension and leaf-wise conformal structure on a fibered $3$--manifold.  This provides a more convenient framework for carrying out the construction.  In Section \ref{S:singular solv structure} we describe the leaf-wise conformal structure coming from the singular-solv structure which is essentially the starting point for our construction.  Since the singular-solv structure is constructed from the axis for a pseudo-Anosov with respect to the Teichm\"uller metric, this explains the appearance of $\Vert \phi \Vert_T$ in the bound we obtain.

Next, in Section \ref{S:Dehn twist and fill} we describe how we will perform Dehn surgery on the manifold $M$, viewed as a suspension, so that the Dehn filled manifold is still a suspension, and the monodromy has been composed with a power of a Dehn twist.  This is followed by Section \ref{S:good solid tori}, which contains two technical lemmas: one describes a particular solid torus in $M$ that is situated nicely with respect to the singular-solv structure; the second produces the explicit solid torus and leaf-wise conformal structure that we will use in our Dehn filling.

We assemble the ingredients in Section \ref{S:bound proof} and prove a more precise version of Theorem~\ref{T : twisted Linch bound 2}.  In Section \ref{S:examples} we explain how to use this to show that $\Psi_{wp}(L)$ contains infinitely many conjugacy classes of pseudo-Anosov mapping classes on every closed, orientable surface of genus at least $2$, and finally in Section \ref{S:generalities} we explain how to apply the theorem to produce examples in a much more general setting.

For the remainder of this section, we will take $\Sigma$ to be an arbitrary surface.  The two cases of interest to us are when $\Sigma$ is a closed surfaces (in which case we often denote it by $S$) and when $\Sigma$ is an annulus (and we denote it $\BA$).

\subsection{Leaf-wise conformal structures} 
\label{S:leaf-wise conformal structure}  
Suppose $\Sigma$ is a surface and $\pi \colon E \to B$ a $\Sigma$--bundle over a connected $1$--manifold $B$ which we view as either an interval in $\R$ or a circle $L\BS^1 = \R / L \BZ$ of length $L > 0$ (in particular, $\R$ locally acts on $B$ by translation).  We consider local flows $\phi_t$ on $E$ such that for all $x \in E$ and $t \in \R$, $\pi(\phi_t(x)) = \pi(x) + t$, as long as $\phi_t(x)$ is defined.  If $E = \Sigma \times J$ for an interval $J \subseteq \R$ and $\pi$ the projection onto the second factor, then after changing the product structure we may assume that $\phi_t(x,s) = (x,s+t)$.
For $s \in B$, we write $\Sigma_s = \pi^{-1}(s)$.  Given $s,t \in B$, we can restrict $\phi_{t-s}$ to a homeomorphism
\[ \phi_{t-s} \colon \Sigma_s \to \Sigma_t.\]

For fibrations $\pi \colon M \to L\BS^1$ and $s,t \in L\BS^1$, $t-s$ is only defined modulo $L \BZ$.  In this situation, we pass to the infinite cyclic cover of $M$ corresponding to the kernel of the homomorphism $\pi_* \colon \pi_1M \to \pi_1L\BS^1 \cong \BZ$, and lift the flow and fibration over $L\BS^1$ to a fibration over $\R$, to well-define $\phi_{t-s}$.  Then we have $\Sigma_0 = \Sigma_L$ in $M$ and $\phi = \phi_L \colon \Sigma_0 \to \Sigma_0$ is the monodromy of the bundle.  We refer to the bundle and flow $(\pi \colon M \to L\BS^1,(\phi_t))$ as a {\em suspension} (since $\phi_t$ is naturally the suspension flow of the monodromy $\phi$).

A {\em leaf-wise conformal structure} $\zeta$ on any $\Sigma$--bundle $E$ is a conformal structure on each surface $\Sigma_s$, making it into a Riemann surface $\Sigma_s^\zeta$, such that $\phi_{t-s} \colon \Sigma_s \to \Sigma_t$ is a quasi-conformal homeomorphism whenever it is defined.
Let $\nu_{t,s} \in \belt_1(\Sigma_s^\zeta) \subset \belt(\Sigma_s^\zeta)$ denote the Beltrami differential of $\phi_{t-s}$.  

If the path of Beltrami differentials $t \mapsto \nu_{t,s} \in \belt_1(\Sigma_s^\zeta)$ is piecewise smooth, we write
\[ \mu_s^\zeta = \frac{d}{dt}\Big|_{t=s}\nu_{t,s}.\]
We call this the {\em tangent field} of the family.  To justify this name, observe that the map $t \mapsto [\nu_{t,s}] \in \CT(\Sigma_s^\zeta) = \CT(\Sigma)$ defines a path in the Teichm\"uller space and that $\mu_t^\zeta \in \belt(\Sigma_t^\zeta)$ represents the tangent vector at time $t$ to this path, for each $t \in J$.

\begin{proposition}\label{P:length bound}
Let $(\pi:M\to L \, \BS^1,(\phi_t))$ be a suspension with fibers $S_t=\pi^{-1}(t)$, let $\zeta$ be a leaf-wise conformal structure, and let $\mu_t^\zeta\in \belt(S_t^\zeta)$ be its tangent field. The mapping class $\phi = \phi_L \in \Mod(S_0)$ has Weil-Petersson translation length $\Vert \phi \Vert_{wp}$ bounded by
$$\Vert \phi \Vert_{wp}\le\int_0^L \left(\sqrt{\int_{S_t}\vert\mu^\zeta_t\sigma^\zeta_t\vert^2}\right)dt,$$
where $\sigma^\zeta_t$ is the hyperbolic metric uniformizing $S_t^\zeta$.
\end{proposition}

\begin{proof}[Proof of Proposition \ref{P:length bound}]
The translation length $\Vert\phi\Vert_{wp}$ is bounded from above by the length of the path 
$$[0,L] \to\CT(S_0),\ t\mapsto [\phi_t \colon S_0\to S_t^\zeta].$$
In formulas this means that 
$$\Vert \phi \Vert_{wp}\le\int_0^L\Vert\mu_t^\zeta\Vert_{wp}.$$
The desired bound follows now from Lemma \ref{lem wp}.
\end{proof}

\subsection{Singular solv structure} 
\label{S:singular solv structure}

Suppose $\phi \colon S \to S$ is a pseudo-Anosov {\em homeomorphism} and $M = M_\phi$ is the mapping torus.  The {\em singular-solv metric} on $M$ is a piecewise Riemannian metric that induces a Euclidean cone metric on the fibers of a fibration $\pi \colon M \to L \BS^1$, where $L = \log(\lambda(\phi)) = \Vert \phi \Vert_T$; see e.g.~\cite{CannonThurston,mcmullen2000polynomial}.  There is a natural unit speed flow $(\phi_t)$ making $(\pi \colon M \to L\BS^1,(\phi_t))$ into a suspension so that:
\begin{enumerate}
\item The Euclidean cone metrics on the fibers defines a leafwise conformal structure $\zeta$,
\item the maps $\phi_{t-s} \colon S_s^\zeta \to S_t^\zeta$ are Teichm\"uller mappings for all $s,t$, with initial and terminal quadratic differentials $\varphi_s$ and $\varphi_t$, respectively,
\item for all $s,t$, the Beltrami differential of $\phi_{t-s}$ is given by $\nu_{t-s} = \tanh(t-s) \frac{\bar \varphi_s}{|\varphi_s|}$,
\item the tangent field is given by $\mu_t^\zeta = \frac{\bar \varphi_t}{|\varphi_t|}$, and
\item the vertical and horizontal foliations of $\varphi_s$ are the stable and unstable foliations for $\phi$, for all $s$.
\end{enumerate}

Note that the tangent field $(\mu^\zeta_t)$ to the singular-solv leaf-wise conformal structure $\zeta$ has $|\mu^\zeta_t| = 1$, for all $t$.  Therefore, applying Proposition~\ref{P:length bound}, we obtain
\[ \Vert \phi \Vert_{wp}\le\int_0^L \left(\sqrt{\int_{S_t}\vert\mu^\zeta_t\sigma^\zeta_t\vert^2}\right)dt \leq \int_0^L \sqrt{\int_{S_t}( \sigma^\zeta_t)^2} = L \sqrt{\Area(S)} = \Vert \phi \Vert_T \sqrt{\Area(S)},\]
as expected.  To prove Theorem~\ref{T : twisted Linch bound 2} we will perform an appropriate Dehn surgery on $M$ by drilling out the curve $\alpha$ on a fiber $S$, and replacing it with an appropriate solid torus and leaf-wise conformal structure.  The goal of the next section is to describe the setup for such a surgery construction.

\subsection{Dehn twists and Dehn filling} 
\label{S:Dehn twist and fill}

For any interval $J \subset \R$, the product
$$\BT_J =\BA\times J,$$
is a (not necessarily compact) solid torus.  We denote the local flow in this special case by $\phi^\BT_t(x,s)=(x,t+s)$ (defined for $s,t+s \in J$, as usual), and denote the projection onto the second factor by $\pi_\BT:\BT_J \to J$.  In particular, we have $\pi_\BT \phi^\BT_t (x,s) = \pi_\BT(x,s) + t = s+ t$, where defined.  As usual, we define $\BA_s = \BA \times \{s \}$ for all $s \in J$.

If we have a suspension $(\pi:M\to L\BS^1,(\phi_t))$, an embedding of a compact solid torus $\iota:\BT_J\hookrightarrow M$, and an orientation preserving local isometry $\bar \iota \colon J \to L \BS^1$, then we say that the suspension and solid torus are {\em compatible (via $\iota$ and $\bar \iota$)} if the following diagram commutes, whenever all maps are defined
\[ \xymatrix{ 
\BT_J \ar[r]^{\phi^\BT_s} \ar[d]_\iota & \BT_J \ar[r]^{\pi^\BT} \ar[d]_\iota & J \ar[d]^{\bar \iota}\\
M \ar[r]^{\phi_s} & M \ar[r]^\pi & L\BS^1.\\
} \]

Given a compatible suspension $(\pi \colon M \to L\BS^1,(\phi_t))$ and solid torus $\iota \colon \BT_J \to M$ with $J = [a,b]$, let $M' = M - \iota(\mathrm{int}(\BT_J))$.  For each $k$, the $\frac1k$--Dehn surgery on $\iota(\BT_J)$ in $M$ is obtained gluing $\BT_J$ to $M'$ via a homeomorphism $g_k \colon \partial \BT_J \to \partial M' = \iota(\partial \BT_J)$,
given by 
\[ g_k(x) = \left\{ \begin{array}{ll} \iota(x) & \mbox{for } x \in \BA \times \{a\} \cup \partial \BA \times J\\
\iota \circ t_k(x) & \mbox{for } x \in \BA \times \{b\}. \end{array} \right. \]
where $t_k \colon \BA \to \BA$ is a homeomorphism representing the $k^{th}$ power of a Dehn twist in the core curve of $\BA$.
The resulting manifold $M_k = M' \cup_{g_k} \BT_J$ admits a flow $(\phi_s)$ which locally agrees with the flow of the same name on $M' \subset M$ and on $\BT_J$ locally agrees with $(\phi_s^{\BT})$, since the original solid torus was compatible.
The monodromy is conjugate (by a power of $\phi$) to the composition $T_\alpha^k \circ \phi$, where $\alpha$ is a curve in $S_0$ in the isotopy class of the core of $\iota(\BT_J)$; see e.g.~\cite{Stallings-Fib,Harer-Fib,Long-Morton-Fib}.

\subsection{Good solid tori} 
\label{S:good solid tori}

Throughout this section, we let $M = M_\phi$ for a pseudo-Anosov $\phi$.
We will perform Dehn surgeries as described in the previous section in the presence of leaf-wise conformal structures on both $M$ and the filling solid tori $\BT_J$.  The leaf-wise conformal structure on $M$ will come from the singular-solv metric, and the solid torus which we will remove is described in Lemma~\ref{L:good solid torus} below.  The leaf-wise conformal structure on $\BT_J$ will be constructed by hand and will be such that the gluing maps restricted to the top and bottom $\BA \times \partial J$ are conformal.  This is described in Lemma~\ref{L:sharp leafwise annulus}.  As the proofs of these are somewhat lengthy and technical, we defer their proofs to Section~\ref{S:constructions and estimates}.

Recall that for any simple closed curve $\alpha \subset S$ the {\em twisting coefficient of $\alpha$ for $\phi$}, denoted $\tau_\alpha(\phi)$, is defined to be the distance in the subsurface projection to the annulus with core curve $\alpha$ of the stable and unstable laminations $\mathcal L_+,\mathcal L_-$ for $\phi$:
\[ \tau_\alpha(\phi) =  d_\alpha(\mathcal L_+,\mathcal L_-), \]
see \cite{MM2} for details.  As in Theorem~\ref{T : twisted Linch bound 2} we now assume $\tau_\alpha = \tau_\alpha(\phi) \geq 9$.

The importance of this condition rests on a result of Rafi \cite{rafi2005characterization} which provides a definite modulus Euclidean annulus (depending on $\tau_\alpha$) isometrically embedded in the Euclidean cone metric $\varphi_t$ defining the conformal structure $\zeta_t$ on $S_t$ for some $t$.  Using $\phi_t$ to flow this backward and forward produces the required solid torus.  Carrying out this construction carefully leads to the following, whose proof we defer to Section~\ref{S:singular solv tori}.

\begin{lemma} \label{L:good solid torus} For some $h \geq \tfrac{1}2 \arccosh(\tfrac{\tau_\alpha}2-3)$ and $J = [-h,h]$ there is an embedding $\iota \colon \BT_J \to M$ compatible with the suspension $(\pi \colon M \to L\BS^1,(\phi_t))$ which is disjoint from the singularities.
Furthermore, the induced leaf-wise conformal structure $\zeta$ on $\BT_J$ agrees with the standard one on $\BA_{-h}$, and there exists  $\Psi \colon \BA_h^\zeta \to \BA_{-h}^\zeta$ conformal so $\Psi \circ \phi_{2h}^\BT\colon \BA_{-h} \to \BA_{-h}$ is the $r^{th}$ power of a Dehn twist for some integer $r$.
\end{lemma}

To fill $M' = M \setminus \iota(\mathrm{int}(\BT_J))$ with $\BT_J$ affecting an arbitrary Dehn twist in $\alpha$ as described in the previous subsection, we will need a different leaf-wise conformal structure on $\BA_J$.  To do this we essentially follow the same idea as in the proof of incompleteness of the Weil-Petersson metric by Wolpert \cite{Wol-incomplete} and Chu \cite{chu}.  Specifically, we recall that the incompleteness comes from paths in Teichm\"uller space exiting every compact set in which a curve on the surface is ``pinched" to have length tending to zero, but which nonetheless has finite WP-length.  We apply this idea to first pinch $\alpha$ to be arbitrarily short along the first half of the interval $J$, then perform as much twisting as we like on a negligible sub-interval in the middle of $J$, and then ``un-pinch''.  By carrying out all the estimates using the complete (infinite area) hyperbolic structure on the interior of the annulus, we may apply the Scwartz-Pick Theorem to obtain upper bounds whenever the annulus is conformally embedded into a Riemann surface.   The specific construction we use provides the following. It will be proven in Section \ref{sec:gluing_tori}.

\begin{lemma} \label{L:sharp leafwise annulus} Given $k \in \mathbb Z$, $h > 0$, $\rho > 1$, there exists a leaf-wise conformal structure $\eta_k$ on the solid torus $\BT_J = \BA \times J$ for $J = [-h,h]$ agreeing with the standard structure (of modulus $1$) on $\BA_{-h} = \BA$, with tangent field $(\mu_s^{\eta_k})$ identically zero in a neighborhood of $\partial \BA \times J$, so that for all $s \in J$,
\[ \int_{\BA_s^{\eta_k}} \vert \mu_s^{\eta_k} \sigma_s^{\eta_k} \vert^2 \leq \frac{8 \rho \pi}{h^2},\]
where $\sigma_s^{\eta_k}$ is the complete hyperbolic metric on the interior of $\BA_s$, and so the local flow $(\phi_s^\BT)$ restricted to the boundary circles are dilations.
Moreover, there is a conformal map $T_k \colon \BA_{h}^{\eta_k} \to \BA_{-h}^{\eta_k}$ so that the composition $T_k \circ \phi_{2h}^\BT \colon \BA_{-h}^{\xi_k} \to \BA_{-h}^{\xi_k}$ is the $k^{th}$ power of a Dehn twist in the core curve.
\end{lemma}

\subsection{Proof of  Theorem~\ref{T : twisted Linch bound 2}.} 
\label{S:bound proof}

This theorem will follow easily from the following more precise version by setting $c = 2\pi(1 + \frac6{h^2})$.

\begin{theorem} \label{T : twisted Linch bound 3}
Suppose $\phi \colon S \to S$ is a pseudo-Anosov on a closed surface, $\alpha \subset S$ is a simple closed curve with $\tau_\alpha \geq 9$, and $k \in \mathbb Z$.  Then
\[
 \Vert T_\alpha^k \circ \phi \Vert_{wp} \leq \Vert \phi \Vert_T \sqrt{2 \pi |\chi(S)| \left(1+\frac{6}{h^2} \right)}
\]
for some $h \geq \tfrac12 \arccosh \left( \tfrac{\tau_\alpha}2 - 3 \right)$.
\end{theorem}

\begin{proof} Suppose $\iota \colon \BT_J \to M$ is the solid torus from Lemma~\ref{L:good solid torus} and $\Psi \colon \BA_h \to \BA_{-h}$ and $r$ as in Lemma~\ref{L:good solid torus}.  Fix $k \in \BZ$, $\rho >  1$, and let $\eta_k$ be the leaf-wise conformal structure on $\BT_J$, and $T_k \colon \BA_h \to \BA_{-h}$ as in Lemma~\ref{L:pinch-twist}.  We will prove the required bound for $\Vert T_\alpha^{k-r} \circ \phi \Vert_{wp}$, which will suffice since $k$ was arbitrary.

Let $M' = M - \iota(\mathrm{int}(\BT_J))$ and glue $\BT_J$ to $M'$ along their boundaries via the map
\[ g_k \colon \partial \BT_J \to \partial M' = \iota(\partial \BT_J), \]
given by 
\[ g_k(z,t) = \left\{ \begin{array}{ll} \iota(z,t) & \mbox{for } (z,t)  \in \BA \times \{-h\} \cup \partial \BA \times J\\
\iota \circ \Psi^{-1} \circ T_k(z,t) & \mbox{for } (z,t) \in \BA \times \{h\}. \end{array} \right. \]
The resulting manifold $M_k = M' \cup_{g_k} \BT_J$ admits a flow $(\phi_s)$ which locally agrees with the flow of the same name on $M' \subset M$ and on $\BT_J$ locally agrees with $(\phi_s^{\BT})$, since the original solid torus was compatible.  Let $\iota_k \colon \BT_J \to M_k$ denote the compatible inclusion of the solid torus $\BT_J$.

We claim that the leaf-wise conformal structures $\zeta$ on $M$ (restricted to $M'$) and $\eta_k$ on $\BT_J$ glue together to give a leaf-wise conformal structure $\zeta_k$ on $M_k$.  Near $\BA_{- h}$ in the fiber containing this annulus, $\BA_{-h}^\zeta$ and $\BA_{-h}^{\eta_k}$ are the standard structures, while near $\BA_h$ in the fiber containing it, we note that the map $\Psi^{-1} \circ T_k \colon \BA_h^{\eta_k} \to \BA_h^{\zeta}$ is the composition of two conformal maps, hence is conformal.  Near all other annuli, we have removed a locally isometrically embedded flat cylinder and glued in another one.  Since $(\phi_s^\BT)$ is a dilation on the boundary for both $\zeta$ and $\eta_k$, the gluing maps of boundaries are in fact dilations, and the conformal structures can be explicitly glued together.  Let $(\mu^{\zeta_k}_s)$ denote the tangent field, which is given by $\mu^\zeta_s$ on $S_t \cap M'$ and by $\mu^{\eta_k}_s$ on $S_s \cap \iota_k(\BT_J)$.

The monodromy of $M_k$ is given by $T_\alpha^{k-r} \circ \phi$.  To see this, note that conjugating $\Psi^{-1} T_k$ by $\phi_{2h}^\BT$ we get
\[ (\phi_{2h}^\BT)^{-1}\Psi^{-1} T_k \phi_{2h}^\BT  = (\Psi \phi_{2h}^\BT)^{-1} (T_k \phi_{2h}^\BT) \]
which is the composition of the $k^{th}$ power of a Dehn twist and the inverse of the $r^{th}$ power of a Dehn twist, both in the core curve of the annulus.  Thus we have changed the original monodromy by the $(k-r)^{th}$ power of a Dehn twist in the core curve of $\BA_h$, which is the image of $\alpha$ by the flow.  Hence the monodromy is conjugate (by a power of $\phi$) to $T_\alpha^{k-r}\phi$.

For any $s \in L\BS^1$, set $S_s^0 = S_s \cap M'$.  The closure of the complement of $S_s^0$ in $S_s$ is a disjoint union of annuli (slices of the product $\BT_J = \BA \times J$) which we denote
\[ \overline{S_s - S_{s,0}} = \bigsqcup_{i=1}^{n(s)} \BA_{s,i}.\]
The number of annuli $n(s)$ is bounded by $\frac{3}{2}|\chi(S)|$, which is the largest number of disjoint essential annuli on $S$ that are pairwise nonhomotopic.

Now we write the integral as a sum of integrals on the subsurfaces
\[ \int_{S_s^{\zeta_k}} \vert \mu^{\zeta_k}_s \sigma^{\zeta_k}_s \vert^2 = \int_{S_{s,0}^{\zeta_k}} \vert \mu^{\zeta_k}_s \sigma^{\zeta_k}_s \vert^2 + \sum_{i=1}^{n(s)} \int_{\BA_{s,i}^{\zeta_k}} \vert \mu^{\zeta_k}_s \sigma^{\zeta_k}_s \vert^2. \]
The flow in $M$ has unit Teichm\"uller speed, and so $\vert \mu^\zeta_s \vert = 1$.  Since $\mu^{\zeta_k}_s = \mu^\zeta_s$ on $S_{s,0}^{\zeta_k} = S_{s,0}^\zeta$, the first integral is bounded by the hyperbolic area (as it was with $M$ itself)
\[ \int_{S_{s,0}^{\zeta_k}} \vert \mu^{\zeta_k}_s \sigma^{\zeta_k}_s \vert^2 \leq \Area(S_{s,0}^{\zeta_k}) \leq \Area(S_s^{\zeta_k}) = 2\pi |\chi(S)|.\]
Each annulus $\BA_{s,i}$ is the $\iota_k$--image of $\BA_{s(i)}$ where $\bar \iota_k(s(i)) = s$, and $\iota_k$ induces a conformal map $\BA_{s(i)}^{\eta_k} \to \BA_{s,i}^{\zeta_k}$.  By the Schwartz-Pick theorem, the hyperbolic metric $\sigma_s^{\zeta_k}$ on $S_s^{\zeta_k}$ restricted to $\BA_s^{\zeta_k}$ is bounded above by the complete hyperbolic metric $\sigma_{s(i)}^{\eta_k}$ on $\BA_{s(i)}^{\eta_k} \cong \BA_{s,i}^{\zeta_k}$.  Combining this with Lemma~\ref{L:pinch-twist} implies that for each $i = 1,\ldots, n(s)$,
\[ \int_{\BA_{s,i}^{\zeta_k}} \vert \mu^{\zeta_k}_s \sigma^{\zeta_k}_s \vert^2 \leq \int_{\BA_{s(i)}^\zeta} \vert \mu_{s(i)}^{\eta_k} \sigma_{s(i)}^{\eta_k} \vert^2 \leq \frac{8 \rho \pi}{h^2}.\]

Now, by Proposition~\ref{P:length bound} and since $n(s) \leq \frac32|\chi(S)|$, we have
\begin{eqnarray} \label{eq:twistbound}
\Vert T_\alpha^{k-r} \circ \phi \Vert_{wp} & \leq & \displaystyle{\int_0^L \sqrt{\int_{S_s^{\zeta_k}} \vert \mu^{\zeta_k}_s \sigma^{\zeta_k}_s \vert^2} ds \leq \int_0^L \sqrt{2\pi |\chi(S)| + \sum_{i=1}^{n(s)} \frac{8 \rho \pi}{h^2}} \, \,  ds} \\ 
& \leq & \displaystyle{\Vert \phi \Vert_T \sqrt{2 \pi |\chi(S)| + \frac{8 \rho \pi }{h^2} \left(\frac{3}{2} |\chi(S)| \right)  } }  \notag
\end{eqnarray}
Taking the limit as $\rho \to 1$, we obtain
\begin{eqnarray*} \Vert T_\alpha^{k-r} \circ \phi \Vert_{wp} & \leq & \displaystyle{\Vert \phi \Vert_T \sqrt{2 \pi |\chi(S)| \left(1+\frac{6}{h^2} \right)} }.
\end{eqnarray*}
This completes the proof.
\end{proof}

We also record the following corollary of Theorem~\ref{T : twisted Linch bound 2}.

\begin{corollary} \label{C: basic Teich to wp} If $\Vert \phi \Vert_T \leq \frac{c'}{|\chi(S)|}$ for some $c' >0$, and if $\alpha$ is a curve with $\tau_\alpha(\phi) \geq 9$, then for all $k \in \mathbb Z$
\[ \Vert T_\alpha^k \circ \phi \Vert_{wp} \leq \frac{c'\sqrt{c}}{\sqrt{|\chi(S)|}}.\]
where $c$ is the constant from Theorem~\ref{T : twisted Linch bound 2}.\qed
\end{corollary}

\subsection{Examples in all genus.} 
\label{S:examples}

From Corollary~\ref{C: basic Teich to wp} we can prove Corollary~\ref{C: infinitely many}
from the introduction, which we restate here, with an explicit bound on $L$.

\begin{corollary} \label{C:specific infinitely many} For any $L \ge 124$
the set $\Phi_{wp}(L)$ contains infinitely many conjugacy classes of pseudo-Anosov mapping classes for every closed surface of genus $g \geq 2$.
\end{corollary}

\begin{proof} Consider the curves $\alpha,\alpha',\beta$ on a genus two surface $S_2$ shown in Figure~\ref{F:genus 2 example}.  Suppose $\phi_2 \colon S_2 \to S_2$ is the mapping class defined by $\phi_2 = T_{\alpha'} T_{\alpha}^{9}T_\beta^{-1}$.  We can (explicitly) construct a square-tiled flat metric on $S$ so that $\beta$ is vertical and $\alpha,\alpha'$ are horizontal, and so that $\phi_2$ is affine (c.f.~Thurston's construction \cite{Th}).  Specifically, this surface is built from a cylinder of height 9 and circumference 2 about $\alpha$ and height $1$ and circumference $2$ about $\alpha'$ as shown in the middle of the figure.  This can be done so that the derivative of $\phi_2$ is given by the matrix on the right of the figure.

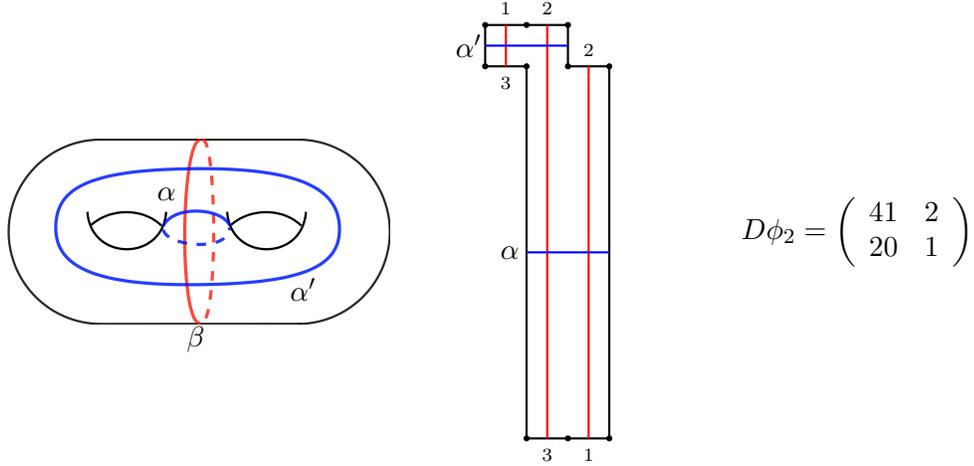
\begin{figure}[htb]
\begin{center}
\begin{tikzpicture}[xscale=1.1,yscale=1.1]
\node at (0,0) {\includegraphics[scale=.2]{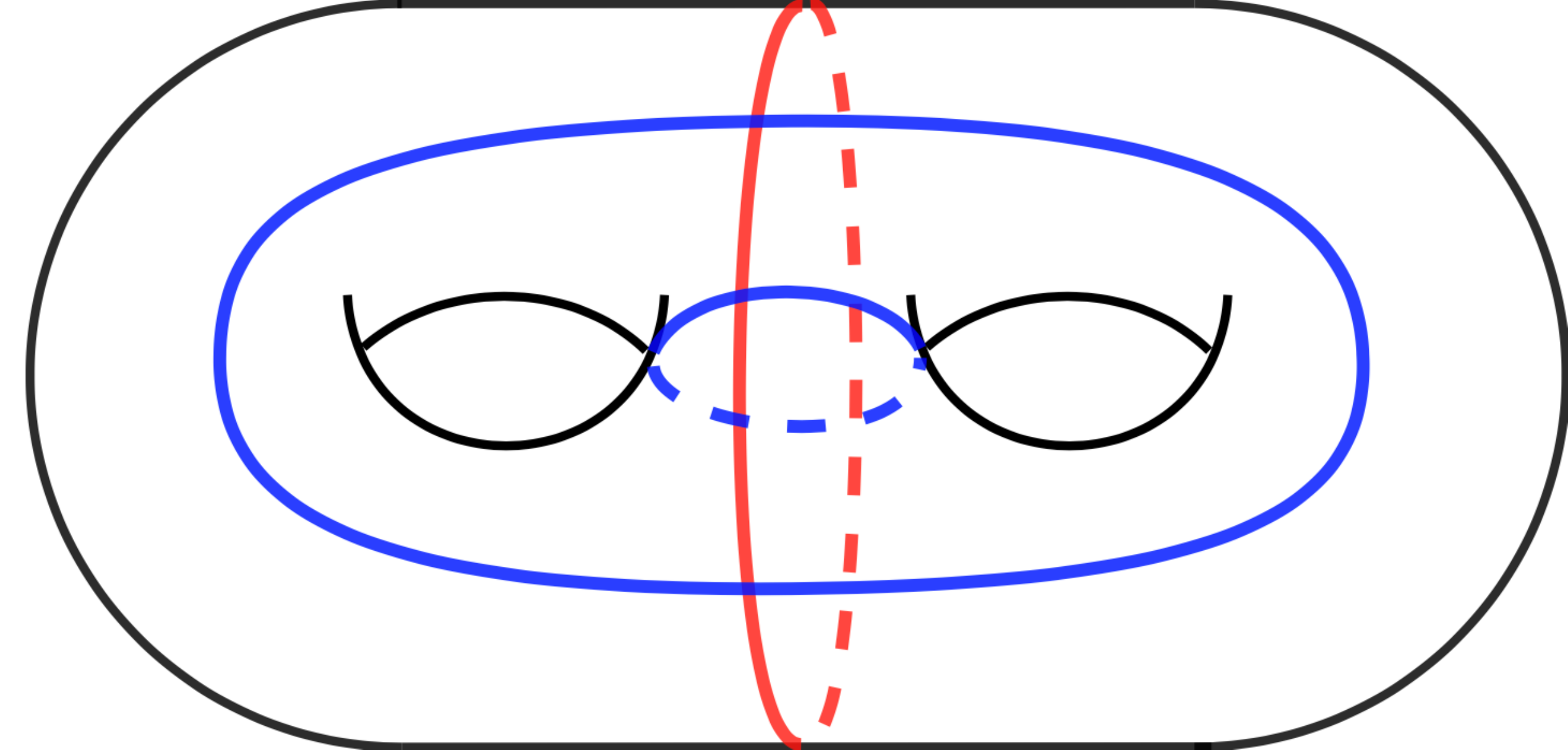}};
\node at (8,0) {$\displaystyle{D\phi_2 = \left( \begin{array}{cc} 41 & 2 \\ 20 & 1 \end{array} \right)}$};
\node at (0,-1.3) {$\beta$};
\node at (-.35,.45) {$\alpha$};
\node at (1.3,-.7) {$\alpha'$};
\draw[thick] (4,-2.5) -- (5,-2.5) -- (5,2) -- (4.5,2) -- (4.5,2.5) -- (3.5,2.5) -- (3.5,2) -- (4,2) --  (4,-2.5);
\draw[thick , red] (4.25,-2.5) -- (4.25,2.5);
\draw[thick , red] (4.75,-2.5) -- (4.75,2);
\draw[thick , red] (3.75,2) -- (3.75,2.5);
\draw[thick , blue] (4,-.25) -- (5,-.25);
\draw[thick , blue] (3.5,2.25) -- (4.5,2.25);
\filldraw (4.5,-2.5) circle (.03cm);
\filldraw (4,-2.5) circle (.03cm);
\filldraw (5,-2.5) circle (.03cm);
\filldraw (4.5,2) circle (.03cm);
\filldraw (4,2) circle (.03cm);
\filldraw (5,2) circle (.03cm);
\filldraw (3.5,2) circle (.03cm);
\filldraw (4.5,2.5) circle (.03cm);
\filldraw (4,2.5) circle (.03cm);
\filldraw (3.5,2.5) circle (.03cm);
\node at (3.8,-.25) {$\alpha$};
\node at (3.3,2.25) {$\alpha'$};
\node at (4.25,2.7) {\tiny 2};
\node at (4.75,2.2) {\tiny 2};
\node at (3.75,1.8) {\tiny 3};
\node at (4.25,-2.7){\tiny 3};
\node at (3.75, 2.7){\tiny 1};
\node at (4.75,-2.7){\tiny 1};
\end{tikzpicture}
\caption{\label{F:genus 2 example} Left: curves on a genus $2$ surface.  Middle: square tiled surface built by isometrically identifying vertical sides by isometry in the obvious way and horizontal sides as indicated by the numbering.  The tall rectangle has height $9$ and width 2, while the smaller one has height $1$ and width $2$.  The horizotonal curves are $\alpha$ and $\alpha'$ as labeled, and the vertical curve is $\beta$.  Right: Derivative of the pseudo-Anosov $\phi_2$.}
\end{center}
\end{figure}

From this we compute that $\lambda(\phi_2) = 21+2\sqrt{110}$, and the eigenspaces of $D\phi_2$, which define the stable/unstable foliations for $\phi_2$, have slopes $\sqrt{110} - 10 \leq \frac12$ and $-\sqrt{110} - 10 \leq -20$.  From the description of the cylinder about $\alpha$, we see that $\tau_{\alpha} \geq 11$.

Now let $M$ be the mapping torus of $\phi_2$.  Since $\alpha$ and $\beta$ intersect twice, as was shown in \cite[Lemma~3.8-3.9]{ALM}, there is a genus $2$ surface $S' \subset M$ which contains $\alpha$ and is transverse to the suspension flow of $\phi_2$.  Thus $S'$ represents a class $[S']$ in the closure of the cone $\mathcal C = \R_+ F$ on the fibered face $F$ of the Thurston norm ball ${\bf B}$ containing the class $[S]$ of $S$.  In fact, it was also shown in \cite{ALM} that $[S]$ and $[S']$ are linearly independent.   It follows that for all $g \geq 3$, the class $(g-2)[S] + [S']$ is represented by a surface $S_g$ which is a fiber in a fibration of $M$ over $S^1$ and contains $\alpha$.  By linearity of the Thurston norm $\mathfrak n$, $-\chi(S_g) = \mathfrak n(S_g) = 2g-2$, and since $[S_g]$ is primitive, $S_g$ is connected of genus $g$;  see Theorem~\ref{T:fibered face}.  Moreover, since both $S$ and $S'$ contain $\alpha$, we can realize $S_g$ so that it also contains $\alpha$.  

Let $\phi_g \colon S_g \to  S_g$ denote the monodromy.  Recall from Theorem~\ref{T:entropy extension} that $\mathfrak h \colon \mathcal C \to \mathbb R_+$, which extends $\Vert \cdot \Vert_T$, is convex and homogeneous of degree $-1$ and thus
\[ \Vert \phi_3 \Vert_T = \mathfrak h([S_3]) \leq \mathfrak h([S_2]) = \Vert \phi_2 \Vert;\]
see \cite[Lemma~3.11]{ALM}.  Applying convexity again we have
\[ \vert \chi(S_g) \vert \Vert \phi_g \Vert_T \leq 4 \log(\lambda(\phi_2)) = 4 \log(21+2 \sqrt{110}) \leq 14.95.\]
The twisting numbers $\tau_\alpha(\phi_g)$ and $\tau_\alpha(\phi_2)$ are equal:
this is because these are twisting numbers for the
monodromy maps of fibers $[S]$ and $[S_g]$ in the same fibered face of the Thurston norm
ball. As observed in \cite{Minsky-Taylor}, the  universal covers of $S$ and
$S_g$ can both be identified with the leaf space of the suspension flow in the 
universal cover of $M$, the stable/unstable laminations of $\phi_2$ and $\phi_g$ have
identical lifts in this cover, and $\tau_\alpha$ is computed using this data in the
annular quotient associated to the conjugacy class $[\alpha]$ in $\pi_1(M)$.

In particular $\tau_\alpha(\phi_g) = \tau_\alpha(\phi_2) \geq 11$.  Therefore, by Corollary~\ref{C: basic Teich to wp}, we have
\[ \Vert T_\alpha^k \circ \phi_g \Vert_{wp} \leq \sqrt{c} \frac{14.95}{\sqrt{\vert \chi(S_g)\vert}} \leq \frac{124}{\sqrt{\vert \chi(S_g) \vert }},\]
where we have used $c =2\pi \left(1+\frac{6}{h^2}\right)$ with $h = \tfrac12 \arccosh \left( \tfrac{11}2 - 3 \right)$.
It follows that $\Phi_{wp}(L)$ contains infinitely many pseudo-Anosov mapping classes on every surface of genus $g \geq 2$ for $L \geq 124$.
\end{proof}

\subsection{Generalities} 
\label{S:generalities}

This construction is much more robust, as the following corollary shows.  Given a closed curve $\alpha$ in $M$, there is a subspace $V_\alpha \subset H^1(M)$ consisting of cohomology classes that evaluate to zero on $\alpha$.

\begin{corollary}  \label{cor:fiber_twist}
Suppose $M$ is a closed fibered $3$--manifold with pseudo-Anosov monodromy $\phi \colon S \to S$, and suppose that $\alpha$ is a curve in $S$.  Then there is a constant $c_0 > 0$ and an $\R_+$--invariant neighborhood $U$ of $[S] \in H^1(M)$ with the following property.  If $S' \subset M$ is another fiber of $M$ with $[S'] \in U \cap V_\alpha$, then $\alpha$ is isotopic into $S'$, and for all $k \in \BZ$, the monodromy $\phi' \colon S' \to S'$ satisfies
\[ \Vert T_\alpha^k \circ \phi' \Vert_{wp} \leq \frac{c_0}{\sqrt {|\chi(S')|}}.\]
\end{corollary}

Observe that this corollary does not require any assumption on the twisting coefficient of $\alpha$, though one loses explicit control on $c_0$ because of this.

\begin{proof}
Since $S$ is a fiber we can represent $[S]\in H_2(M) \cong H^1(M)$ as a  closed 1-form
$\omega$ which is nowhere zero.  Fixing a Riemannian metric on $M$, for each $\epsilon$
let $U_\epsilon$ denote a neighborhood of $[S]$ with the property that every element in
$U_\epsilon$ is represented by a closed nowhere-zero $1$-form $\omega'$ with
$|\omega'-\omega|<\epsilon$.
The primitive integral classes in $\R_+U_\epsilon$ are precisely the classes in
$\R_+U_\epsilon$ dual to fibers of a fibration.

We would like to say that
there is a regular neighborhood $N$ of $\alpha$  and an $\epsilon>0$ such that
when $S'$ is a fiber with $[S']\in \R^+ U_\epsilon \cap V_\alpha$ then after an isotopy 
both $S$ and $S'$ meet $N$ in an annulus containing $\alpha$.

First choose a regular neighborhood $W$ which can be written in the form $A\times(-b,b)$
where $A_0 = A\times\{0\}$ is a regular neighborhood of $\alpha$ in $S$ and $\omega|_W = dt$,
where $t$ is the coordinate for $(-b,b)$. 
If $\omega'$ is the 1-form representing $[S']$ then, since $[S']\in V_\alpha$, we have
$\int_\alpha\omega'=0$ so $\omega'$ is exact in $W$.  For suitably small $\epsilon$ we can
write $\omega' = dh$  on $W$. 

Thus for small $\epsilon$ the level sets of $h$ are surfaces transverse
to the $t$ direction. It follows that there is a smaller regular neighborhood $N$ of $\alpha$
so that for each $\omega'$ representing a point in  $U_\epsilon$
there is a compactly supported isotopy in $W$
(moving along vertical lines) which takes $\omega'$ to $dt$ in $N$. 
The integral manifold of (the isotope of) $\ker\omega'$ passing through $\alpha$ contains the annulus
$A_0\intersect N$, and after another isotopy any other intersections of it with $N$ can be
pushed off. This gives the desired fiber $S'$. 

Now in $H_1(\boundary N)$ we have a basis $\mu,\lambda$ where $\mu$ is the boundary of a
meridian disk and $\lambda $ is represented by a component of $\boundary (A_0\intersect N)$.
Let $\cM = M\ssm N$.  For $j\in \mathbb Z$, the Dehn filling of $\cM$  associated to $\mu + j\lambda$ gives a
manifold $M_j$ that can be simultaneously described as the mapping torus of $T^j_\alpha \phi$, and
as the mapping torus of ${T'_\alpha}^j \phi'$, where $\phi$ is the monodromy of the
fibering of $S$, $\phi'$ is the monodromy of the fibering of $S'$, and $T_\alpha$ and
$T'_\alpha$ are the Dehn twist of $\alpha$ in $A_0\intersect N$, considered as a homeomorphism of $S$
and $S'$ respectively.

We want to relate homology classes in $M$ to those in $M_j$. 
Note that the inclusion map
$H_2(\cM)\to H_2(M)$ is injective (since $\boundary\cM = \boundary N$ has one component),
and that its image is exactly $V_\alpha$ (this is an exercise in Poincar\'e duality).
Thus we can invert it on $V_\alpha$ and compose with the inclusion $H_2(\cM) \to H_2(M_j)$ to obtain a linear map
$V_\alpha\to H_2(M_j)$. Let $\FF_j$ denote the fibered face of the Thurston norm ball in
$H_2(M_j)$ containing the image $[S_j]$ of $[S]$. Then $\R_+\FF_j$ is open, and we 
may choose $\epsilon$ sufficiently small that it contains the 
closure of the image of $\R_+U_\epsilon\intersect V$, which we denote $K_j$.
Let $[S'_j]$ denote the image of $[S']$ in $H_2(M_j)$.

Now let $U = \R^+U_\epsilon$, 
and choose $j>0$ so that $T^j_\alpha \phi$ is pseudo-Anosov and $\tau_\alpha(T^j_\alpha\phi) \ge 9$. 

For any $S'$ such that $[S']\in U\intersect V_\alpha$,
as observed above in the proof of Corollary \ref{C: basic Teich to wp},
the twisting numbers $\tau_\alpha(T^j_\alpha\phi)$ and $\tau_\alpha({T'_\alpha}^j\phi')$
are equal since $[S_j]$ and $[S'_j]$ are in the same fibered face $\FF_j$. 

In particular $\tau_\alpha({T'_\alpha}^j\phi')\ge 9$ and
we can therefore apply Theorem~\ref{T : twisted Linch bound 2} to ${T'_\alpha}^j \phi'$, obtaining
\begin{align*}
  \Vert {T'_\alpha}^k \circ \phi' \Vert_{wp}  &= \Vert {T'_\alpha}^{k-j} \circ {T'_\alpha}^j \phi'
  \Vert_{wp} \\
  &\le \Vert {T'_\alpha}^j \phi' \Vert_T \sqrt{c |\chi(S')|}.
\end{align*}
To complete the argument we need to obtain a bound on 
$\Vert {T'_\alpha}^j \phi' \Vert_T$. 
As described in Section \ref{S:Thurston-Fried}, by Theorem~\ref{T:entropy extension} there is a $c_j \ge 1$ depending only on $K_j$
such that if $S''$ is a fiber of $M_j$ such that $[S''] \in K_j$ with monodromy $\phi_{S''}$ then $|\chi(S'')| \Vert \phi_{S''} \Vert_T \le c_j$.
We conclude that
$$
  \Vert {T'_\alpha}^k \circ \phi' \Vert_{wp} 
\le \frac{c_j\sqrt{c}}{\sqrt{|\chi(S')|}},
$$
as required.
\end{proof}

\section{Constructions and estimates} 
\label{S:constructions and estimates}

In this section, we will describe constructions of some specific leaf-wise conformal structures and estimates for the Weil-Petersson norm of their tangent fields.  We begin with some generalities on reparameterizing and gluing leaf-wise conformal structures.  Next we set up some notation for annuli and carry out some computations of hyperbolic area, before turning to the construction of the required leaf-wise conformal structures on solid tori from Lemma~\ref{L:sharp leafwise annulus} necessary for our main construction.  We end this section with the construction of the solid torus from Lemma~\ref{L:good solid torus} we remove from the original manifold $M$.

\subsection{Reparameterizing and gluing} 
Given a piecewise differentiable, monotone, surjective map $g \colon I \to J$, we define $G \colon \Sigma \times I \to \Sigma \times J$ by $G(x,t) = (x,g(t))$.  If $\zeta$ is a leaf-wise conformal structure on $\Sigma \times J$, then it pulls back via $G$ to one on $\Sigma \times I$ denoted $G^*\zeta$, so that $\Sigma_s^{G^*\zeta} = \Sigma_{g(s)}^\zeta$ as Riemann surface structures on $\Sigma$.  For any $s \in I$, the resulting path of Beltrami differentials in $\belt_1(\Sigma_s^{G^*\zeta}) = \belt_1(\Sigma_{g(s)}^{\zeta})$ is just the composition $t \mapsto \nu_{g(t),g(s)}$, and consequently the tangent field is given by
\[ \mu_s^{G^*\zeta} = g'(s) \mu_{g(s)}^\zeta.\]

Suppose that $\zeta_1$ and $\zeta_2$ are leaf-wise conformal structures on $\Sigma \times J_1$ and $\Sigma \times J_2$, respectively, such that $J_1 \cap J_2 = \{s\}$ and there is a conformal map $g \colon \Sigma_s^{\zeta_1} \to \Sigma_s^{\zeta_2}$.  Then we can {\em glue} together the leaf-wise conformal structures to a single leaf-wise conformal structure $\zeta$ on $\Sigma \times J_1 \cup J_2$ using the conformal map $g$.  More precisely, we define the leaf-wise conformal structure on $\Sigma \times J_1$ to be $\zeta_1$ and on $\Sigma \times J_2$ to be $g^*\zeta_2$.  The path in Teichm\"uller space $J \to \CT(\Sigma)$ is the concatenation of the path $J_1 \to \CT(\Sigma)$ with $J_2 \to \CT(\Sigma)$ after adjusting the marking of the latter by the conformal map $g$.

\subsection{Annuli and solid tori}
As above, we let $L \BS^1 = \mathbb R/L\mathbb Z$ denote a circle of length $L >0$.
Any annular Riemann surfaces with finite modulus can be uniformized by Euclidean metrics of the form $L \BS^1 \times J$, where $J \subset \mathbb R$ is an interval; the conformal modulus is $|J|/L$.  Equivalently, this surface is the quotient
\[ L \BS^1 \times J = \{ z = x+iy \in \mathbb C \mid y \in J \}/\langle z \mapsto z+L \rangle.\]
Note that $L \BS^1 \times J$ is conformally equivalent to $r L \BS^1 \times r J$, for any $r > 0$, and we will write $m \BA$ for any annular Riemann surface with conformal modulus $m$.
For $m=1$, we just write $1 \BA = \BA$, though we also allow $\BA$ to denote a topological annulus.

The {\em middle sub-annulus} of the annulus $2m \BA$ of modulus $2m$ is the sub-annulus of modulus $m$ invariant under the full conformal automorphism group.  More precisely, it is given by
\[ \BS^1 \times \left[-\tfrac{m}2,\tfrac{m}2 \right] \subset \BS^1 \times [-m,m].\]
Although the middle sub-annulus encompasses half the Euclidean area, it's hyperbolic area is inversely proportional to the modulus, as the next lemma shows.
\begin{lemma} \label{L:hyp area half} Suppose $m\BA \subset 2m\BA$ is the {\em middle sub-annulus} and $\sigma$ is the complete hyperbolic metric on the interior of $2m\BA$.  Then
\[ \int_{m\BA} \sigma^2 = \frac{\pi}m.\]
\end{lemma}
\begin{proof} We view the annulus $2m\BA$ as the quotient of
\[ H_m = \{ x+iy \in \mathbb C \mid |y| \leq m\}, \]
by the action of $\langle z \mapsto z+1 \rangle$.  A fundamental domain for the middle sub-annulus is
\[ D_{\tfrac{m}2} = \{x+iy \mid 0 \leq x \leq 1, \, -\tfrac{m}2 \leq y \leq \tfrac{m}2 \} \subset H_{\tfrac{m}2} \subset H_m.\]

To find the hyperbolic metric on the interior of $H_m$, map it to the upper half-plane with the conformal map $z \mapsto ie^{\frac{\pi z}{2m}}$.  The pull-back to $H_m$ of the hyperbolic metric $\frac{|dz|}{Im(z)}$ in the upper-half plane is 
\[ \tilde \sigma = \frac{\pi}{2m\cos(\tfrac{\pi y}{2m})} |dz| = \frac{\pi}{2m}\sec(\tfrac{\pi y}{2m}) |dz|. \]
Therefore, integrating this over $D_{\tfrac{m}2}$ provides the required computation:
\[ \int_{\BA_m} \!\!\! \sigma^2 = \!\! \int_{D_{\tfrac{m}2}} \!\!\! \tilde \sigma^2 = \!\! \int_{D_{\tfrac{m}2}} \!\!\!\! \tfrac{\pi^2}{4m^2} \sec^2(\tfrac{\pi y}{2m}) |dz|^2 \!\!
 = \! \tfrac{\pi^2}{4m^2} \!\! \int_0^1 \!\! \int_{-m/2}^{m/2} \!\!\!\! \sec^2(\tfrac{\pi y}{2m}) \, dy \, dx   \! = \! \frac{\pi}m . \]

\end{proof}

In the next three subsections, we carry out the explicit construction of a leaf-wise conformal structure on solid tori and estimates on the Weil-Petersson norm of the associated tangent fields.  This really divides into three parts: pinching, twisting, and then combining the two.  Some care is necessary with the parameterizations involved since, when using these solid tori in Dehn filling, a fiber will typically meet the solid torus in many copies of the annulus.

\subsection{Pinching construction/estimates}  
Here we explicitly describe the leaf-wise conformal structure on a solid torus that ``pinches the core curve,'' and estimate the Weil-Petersson norm of its tangent field.  Specifically, for $h > 0$ and solid torus $\BA \times [0,h)$ we construct a leaf-wise conformal structure that has modulus $1$ on $\BA_0$, modulus of $\BA_t$ tending to infinity as $t \to h$, and so that the Weil-Petersson norm of the tangent field is bounded by $\frac{\sqrt{8 \pi}}h$ for every $t$.  The desired structure will be obtained from the one in the following proposition by an appropriate reparameterization.

\begin{proposition} \label{P:pinching path} There exists a leaf-wise conformal structure
$\zeta$ on $\BA \times [1,\infty)$ such that
    \begin{enumerate}[(a)]
    \item $\BA^\zeta_t$ is isomorphic to $t\BA$.
      \item The associated flow $\phi_s:\BA^\zeta_t \to \BA^\zeta_{t+s}$ 
        is conformal in neighborhoods of the
        boundaries, and is a dilation on the boundaries.
      \item The tangent field $\mu^\zeta_t$ satisfies $|\mu^\zeta_t| = 1/t$ in the middle annulus of
        $\BA^\zeta_t$, and is identically $0$ outside it. 
    \end{enumerate}
    In particular, we have
    \begin{equation}\label{pinch integral 1}
\int_{\BA_t^\zeta} \vert \mu_t^\zeta \sigma_t^\zeta \vert^2 = \frac{2 \pi}{t^3}, 
    \end{equation}
and therefore
\begin{equation}\label{pinch integral 2}
 \int_1^\infty \sqrt{\int_{\BA_t^\zeta} \vert \mu_t^\zeta \sigma_t^\zeta \vert^2} \, dt =
 \sqrt{8 \pi} .
\end{equation}
\end{proposition}

Note that this construction gives a path of finite Weil-Petersson length that diverges in
Teichm\"uller space, by embedding $\BA_1$ into some Riemann surface and using this family
to deform. This recovers the result, due to Wolpert \cite{Wol-incomplete} and Chu \cite{chu}, that the WP metric is incomplete.

\begin{proof}
To begin, we define a $1$--parameter family of quasi-conformal maps $(\tilde f_t \colon
\mathbb C \to \mathbb C)_{t \geq 1}$, so that,
defining 
\[ H_s = \{x+iy \in \mathbb C \mid |y| \leq s \},\]
\begin{itemize}
\item $\tilde f_t$ is the identity on the real coordinate $x$,
\item $\tilde f_t$ takes $H_{1/t}$ to $H_t$, dilating the $y$ coordinate by $t^2$,
  \item $\tilde f_t$ takes $H_1$ to $H_{2t-1}$ and $H_2$ to $H_{2t}$. 
\end{itemize}

Explicitly, $\tilde f_t$ is defined by
\[ \tilde f_t(x+iy) = \left\{ \begin{array}{ll}
x+i(t^2y) & \mbox{ for } 0 \leq y \leq 1/t,  \\\\
x+i(2t-\frac1y) & \mbox{ for } 1/t \leq y \leq 1, \\\\
x+i(y+2(t-1))& \mbox{ for } 1\leq  y,\\ \end{array} \right.  \]
and extend over $x+iy \in \mathbb C$ with $y < 0$ by reflection.

Now since horizontal translations in $\mathbb C$ commute with $\tilde f_t$ and preserve
$H_s$ for all $s,t$, the map $\tilde f_t$ descends to a map
\[ f_t \colon \BA = H_2/\langle z \mapsto z+4 \rangle \to H_{2t}/\langle z \mapsto z+4 \rangle = t \BA.\]
These maps define a leaf-wise conformal structure $\zeta$ on the (noncompact) solid torus
$\BA \times [1,\infty)$ so that $\BA_t^\zeta = t \BA$ and $\phi^\BT_{t-1} = f_t \colon
  \BA_1^\zeta \to \BA_t^\zeta$.
   
Parts (a) and (b) of the proposition follow immediately from the definition. 
For part (c) and the two integral equalities, we need to compute the tangent field $\mu^\zeta_t$.

For any fixed $t \geq 1$, $s \geq 0$, consider the maps $\tilde f_{t,t+s} = \tilde f_{t+s} \circ \tilde f_t^{-1} \colon \mathbb C \to \mathbb C$.  If  $y \geq t$, then
\[ \tilde f_{t,t+s}(x+iy) = x+i(y+2s),\]
while if $-y \geq t$, then $\tilde f_{t,t+s}(x+iy) = x+i(y-2s)$.

On the other hand, if $0 \leq |y| < t$, then for sufficiently small $s > 0$, we have
\[ \tilde f_{t,t+s}(x+iy) = \tilde f_{t+s}(x+i(y/t^2)) = x+i \left( \frac{(t+s)^2}{t^2} y \right).\]
In particular, if $\tilde \nu_{t,t+s}$ is the Beltrami coefficient of $\tilde f_{t,t+s}$, then for each $t$, its derivative with respect to $s$ at $s = 0$ is
\[ \tilde \mu_t(x+iy) = \left\{ \begin{array}{cl} 0 & \mbox{ for } |y| \geq t \\
-\frac1t & \mbox{ for } |y| < t \end{array} \right.\]

The tangent field $\mu_t^\zeta$ is the descent of $\tilde \mu_t$, so we see that part (c)
holds.  Integrating this over  $\BA_t^\zeta = t\BA$, we have
\[ \int_{\BA_t^\zeta} \vert \mu_t^\zeta \sigma_t^\zeta \vert^2 = \frac1{t^2} \int_{\tfrac{t}2\BA} \sigma_t^2 = \frac1{t^2} \frac{\pi}{t/2} = \frac{2\pi}{t^3},\]
where the second equality is by  Lemma~\ref{L:hyp area half}. This
proves equality (\ref{pinch integral 1}).
 An easy computation proves (\ref{pinch
  integral 2}). 
\end{proof}

We can reparameterize the family of Proposition \ref{P:pinching path} by the moral equivalent of
Weil-Petersson arclength to obtain the following: 

\begin{corollary} \label{C: reparameterized} For any $h > 0$, there exists a differentiable homeomorphism $g \colon [0,h) \to [1,\infty)$ so that if $G \colon \BA \times [0,h) \to \BA \times [1,\infty)$ is the induced map of products and $\xi = G^*\zeta$ the pull-back of the leaf-wise conformal structure $\zeta$ on $\BA \times [1,\infty)$ as above, with tangent field $\mu_s^\xi$, then
\[ \int_{\BA_s^\xi} \vert \mu^\xi_s \sigma^\xi_s \vert^2 = \frac{8\pi}{h^2}.\]
\end{corollary}
\begin{proof} Set
\[ s(t) =  \int_1^t \sqrt{\int_{\BA_u^\zeta} \vert \mu_u^\zeta \sigma_u^\zeta \vert^2} \, du = \int_1^t \frac{\sqrt{2\pi}}{u^{3/2}} \, du = \sqrt{8 \pi}(1-t^{-1/2}).\]
Let $t(s)$ be the inverse (given by $t(s) = \frac{8 \pi}{(\sqrt{8\pi}-s)^2}$) which satisfies
\[ t'(s) = \frac{1}{s'(t(s))} = \frac{t(s)^{3/2}}{\sqrt{2 \pi}}.\]
For any $h > 0$, set $g(s) = t(\sqrt{8\pi} s / h)$, and observe that $g([0,h)) = [1,\infty)$ with
\[ g'(s) = \frac{\sqrt{8\pi}}{h} t'(\sqrt{8\pi} s / h) = \frac{\sqrt{8\pi}}{h} \frac{t(\sqrt{8\pi}s/h)^{3/2}}{\sqrt{2\pi}} = \frac{2g(s)^{3/2}}h.\]
Therefore, setting $G \colon \BA \times [0,h) \to \BA \times [1,\infty)$ the reparameterization and $\xi = G^*\zeta$, we have
\begin{eqnarray*}
\int_{\BA_s^\xi} \vert \mu_s^\xi \sigma_s^\xi \vert^2 & = & \int_{\BA_{g(s)}^\zeta} \vert g'(s) \mu_{g(s)}^\zeta \sigma_{g(s)}^\zeta \vert^2 \\\\
& = & \frac{4g(s)^3}{h^2} \int_{\BA_{g(s)}^\zeta} \vert \mu_{g(s)}^\zeta \sigma_{g(s)}^\zeta \vert^2\\\
& = & \frac{4 g(s)^3}{h^2} \frac{2\pi}{g(s)^3} = \frac{8\pi}{h^2},
\end{eqnarray*}
as required.
\end{proof}

\subsection{Twisting construction/estimates.}  
Here we use an affine twist in the middle sub-annulus and a simple construction to produce a solid torus $\BA \times [-\epsilon,\epsilon]$ with arbitrarily small $\epsilon$, and leaf-wise conformal structure for which the Weil-Petersson norm of the tangent field is also arbitrarily small, and which affects the $k^{th}$ power of the Dehn twist in the core curve.  The trade-off is that the moduli of $\BA_t$ are required to be large. 

\begin{lemma} \label{L:pinch-twist} Given $k \in \BZ$, $\epsilon > 0$, and $\delta > 0$ there exists $m_k > 0$ so that for any $m > m_k$, there exists a leaf-wise conformal structure $\xi_k$ on $\BT_\epsilon = \BA \times[-\epsilon,\epsilon]$ with associated tangent field $(\mu^{\xi_k}_s)$, which is identically zero outside the middle sub-annulus for each $\BA_s^{\xi_k}$, so that for all $s \in [-\epsilon,\epsilon]$, $\BA_s^{\xi_k}$ has modulus $2m$,
\[ \int_{\BA_s^{\xi_k}} \vert \mu^{\xi_k}_s \sigma_s^{\xi_k} \vert^2 < \delta,\]
and the local flow $(\phi_s^\BT)$ restricted to the boundary circles are dilations.

Moreover, there is a conformal map $T'_k \colon \BA_{\epsilon}^{\xi_k} \to \BA_{-\epsilon}^{\xi_k}$ so that the composition $T'_k \phi_{2\epsilon}^\BT \colon \BA_{-\epsilon}^{\xi_k} \to \BA_{-\epsilon}^{\xi_k}$ is the $k^{th}$ power of a Dehn twist in the core curve.
\end{lemma}

\begin{proof}  First, we set
\[ m_k = \sqrt[3]{\tfrac{k^2 \pi}{16\delta \epsilon^2}}.\]
Now fix any $m > m_k$ and let $c = \frac{k}{2\epsilon m}$. We identify $2m\BA$ conformally with 
\[ 2m \BA = \BS^1 \times \left[ -\tfrac{m}2,\tfrac{3m}2 \right].\]
We define $f_s \colon 2m\BA \to 2m\BA$, for all $s \in \R$ by
\[ f_s(x+iy) = \left\{ \begin{array}{ll} 
x+iy & \mbox{for } -\frac{m}2 \leq y \leq 0\\
x+sc y + iy & \mbox{for } 0\leq y \leq m\\
x+ scm + iy & \mbox{for } m \leq y \leq \frac{3m}2,\\ \end{array} \right.  \]
and note that this is an affine shear on the middle sub-annulus $\BS^1 \times [0,m]$, and is conformal outside.
For $s = 2 \epsilon$, the effect of $f_{2 \epsilon}$ on the top boundary component of the annulus is
\[ f_{2 \epsilon}\left(x+i\tfrac{3m}2\right) = (x+ 2 \epsilon c m) + i \tfrac{3m}2 =( x + k) + i \tfrac{3m}2.\]
Since $f_{2\epsilon}$ is the identity on the bottom component, it follows that $f_{2 \epsilon}$ is the $k^{th}$ power of a Dehn twist in the core curve of $2m\BA$.

The maps $(f_t)_{t \in [0,2\epsilon]}$ define a leaf-wise conformal structure on $\BA \times [0,2 \epsilon]$ and by translating the interval back by $\epsilon$ and pulling back, it gives the desired leaf-wise conformal structure $\xi_k$ on $\BA \times [-\epsilon,\epsilon]$ so that $\BA_t$ has modulus $2m$ for all $t$.  Furthermore, the maps $\phi_{t-s}^\BT \colon \BA_s^{\xi_k} \to \BA_t^{\xi_k}$, in product coordinates above, are given by $f_{t-s}$ (because $f$ is already a flow: $f_s \circ f_t = f_{s+t}$).  Thus, the tangent field $(\mu_s^{\xi_k})$ is easily seen to be zero outside the middle sub-annulus and by a computation $\mu_s^{\xi_k} = \frac{ci}2$ in the middle sub-annulus.
Appealing to Lemma~\ref{L:hyp area half} we see that
\[ \int_{\BA_s^{\xi_k}} \vert \mu_s^{\xi_k} \sigma_s^{\xi_k} \vert^2 = \frac{c^2}4 \int_{\BA_s^{\xi_k}} \left( \sigma_s^{\xi_k} \right)^2 = \frac{c^2\pi}{4m} = \frac{k^2 \pi}{16\epsilon^2 m^3} < \frac{k^2 \pi}{16\epsilon^2}\frac1{m_k^3}  = \delta.\]
 Since $f_{2\epsilon}$ is the $k^{th}$ power of a Dehn twist from $2m \BA$ {\em to itself}, the existence of the required map $T'_k$ follows. 
\end{proof}

\subsection{Gluing together solid tori.}  
\label{sec:gluing_tori}

Here we combine the constructions above to produce a single solid torus that first pinches, then twists, then ``un-pinches".  This is obtained by stacking together a sufficiently large piece of the solid torus from Corollary~\ref{C: reparameterized}, a solid torus from \ref{L:pinch-twist}, and then another copy of the first solid torus, but with the reversed flow.  This will prove the following lemma claimed in Section~\ref{S:good solid tori}.

\bigskip

\noindent
{\bf Lemma~\ref{L:sharp leafwise annulus}.} {\em Given $k \in \mathbb Z$, $h > 0$, $\rho > 1$, there exists a leaf-wise conformal structure $\eta_k$ on the solid torus $\BT_J = \BA \times J$ for $J = [-h,h]$ agreeing with the standard structure (of modulus $1$) on $\BA_{-h} = \BA$, with tangent field $(\mu_s^{\eta_k})$ identically zero in a neighborhood of $\partial \BA \times J$, so that for all $s \in J$,
\[ \int_{\BA_s^{\eta_k}} \vert \mu_s^{\eta_k} \sigma_s^{\eta_k} \vert^2 \leq \frac{8 \rho \pi}{h^2},\]
where $\sigma_s^{\eta_k}$ is the complete hyperbolic metric on the interior of $\BA_s$, and so the local flow $(\phi_s^\BT)$ restricted to the boundary circles are dilations.
Moreover, there is a conformal map $T_k \colon \BA_{h}^{\eta_k} \to \BA_{-h}^{\eta_k}$ so that the composition $T_k \circ \phi_{2h}^\BT \colon \BA_{-h}^{\xi_k} \to \BA_{-h}^{\xi_k}$ is the $k^{th}$ power of a Dehn twist in the core curve.}

\bigskip

\begin{proof} The idea of the proof is to first use Corollary~\ref{C: reparameterized} to define $\eta_k$ on $\BA_s$ for $s \in [-h, -\epsilon_k]$, for some $\epsilon_k > 0$ so that the modulus of $\BA_{-\epsilon_k}^{\eta_k}$ is as large as we like.  Then define $\eta_k$ on $\BA_s$, for $s \in [-\epsilon_k,\epsilon_k]$ using Lemma~\ref{L:pinch-twist} to ``do all the twisting".  Finally, we define $\eta_k$ on $\BA_s$ for $s \in [\epsilon_k,h]$, by ``reversing" what was done for $s \in [-h,-\epsilon_k]$.  We now explain the details.

Fix $\epsilon > 0$ sufficiently small so that $h^2 \leq \rho(h-\epsilon)^2$, and note that this implies
\[ \frac{8\pi}{(h-\epsilon)^2} \leq \frac{8 \rho \pi}{h^2}. \]
We will choose $\epsilon_k \in (\epsilon,h)$ and define $\eta_k$ on $\BA \times [-h,h]$ by defining it on three sub-intervals:
\[ J_- = [-h,-\epsilon_k] , \, \, J_0 = [-\epsilon_k,\epsilon_k], \, \mbox{ and } J_+ = [\epsilon_k,h] .\]

First, let $\xi$ be the leaf-wise conformal structure on $\BA \times [0,h-\epsilon)$ from Corollary~\ref{C: reparameterized}.  We can translate the interval $[-h,-\epsilon)$ to $[0,h-\epsilon)$ and pull back to obtain a leaf-wise conformal structure on $\BA \times[-h,-\epsilon)$, also denoted $\xi$, so that for all $s \in [-h,-\epsilon)$
\[ \int_{\BA_s^\xi} \vert \nu^\xi_s \sigma^\xi_s \vert^2 = \frac{8\pi}{(h-\epsilon)^2} \leq \frac{8 \rho \pi}{h^2},\]
as in Corollary~\ref{C: reparameterized}.

Observe that the modulus of $\BA_s^\xi$ tends to infinity as $s$ tends to $-\epsilon$.  Therefore, we may choose $\epsilon_k \in (\epsilon,h)$ so that the modulus of $\BA_{-\epsilon_k}^\xi$ is at least $2m$ where $m > m_k$ and $m_k$ is as in Lemma~\ref{L:pinch-twist} for the given $k$, our chosen $\epsilon$, and for $\delta = \frac{8 \rho \pi}{h^2}$.  Now we describe the leaf-wise conformal structure $\eta_k$ on each of the intervals, explaining how they are glued together.\\

\noindent {\bf The interval $J_-$.}  We define $\eta_k$ to be the restriction of $\xi$ on $J_-$.  From the above we have
\begin{equation} \label{E:bounds final} \int_{\BA_s^{\eta_k}} \vert \nu^{\eta_k}_s \sigma^{\eta_k}_s \vert^2 \leq \frac{8 \rho \pi}{h^2}, \end{equation}
for all $s \in J_-$.\\

\noindent {\bf The interval $J_0$.} Next, let $\xi_k$ be the conformal structure on $\BA \times [-\epsilon,\epsilon]$ as in Lemma~\ref{L:pinch-twist} so that for all $s \in [-\epsilon,\epsilon]$, $\BA_s^{\xi_k}$ has modulus $2m$, and
\[ \int_{\BA_s^{\xi_k}} \vert \mu_s^{\xi_k} \sigma_s^{\xi_k} \vert^2 = \delta. \]
Apply the affine map $J_0 = [-\epsilon_k,\epsilon_k] \to [-\epsilon,\epsilon]$ and let $\eta_k$ be the pulled back conformal structure on $\BA \times J_0$.  Since $\epsilon < \epsilon_k$, it follows that
\[ \int_{\BA_s^{\eta_k}} \vert \mu_s^{\eta_k} \sigma_s^{\eta_k} \vert^2 < \delta = \frac{8 \rho \pi}{h^2},\]
for $s \in J_0$.  That is, (\ref{E:bounds final}) also holds for $s \in J_0$.  Because the modulus of $\BA_s^{\eta_k}$ is $2m$ for each $s \in [-\epsilon_k,\epsilon_k]$, we may glue to the two leaf-wise conformal structure on $\BA \times J_-$ and $\BA \times J_0$.\\

For the final interval, let $T'_k \colon \BA_{\epsilon_k}^{\eta_k} \to \BA_{-\epsilon_k}^{\eta_k}$ be the conformal map from Lemma~\ref{L:pinch-twist} so that $T'_k \phi_{2\epsilon_k} \colon \BA_{-\epsilon_k}^{\eta_k} \to \BA_{-\epsilon_k}^{\eta_k}$ is the $k^{th}$ power of a Dehn twist in the core curve of $\BA_{-\epsilon_k}^{\eta_k}$.\\

\noindent {\bf The interval $J_+$.}  On this final interval, the path of beltrami differentials is the one on $J_-$, ``run backward and remarked by $T'_k$''.  More precisely, we define $\eta_k$ on $\BA_{\epsilon_k+s}$, for $0 \leq s \leq h-\epsilon_k$ so that the map
\[ \phi_{-s}^\BT T'_k \phi_{-s}^\BT \colon \BA_{\epsilon_k+s}^{\eta_k} \to \BA_{-\epsilon_k-s}^{\eta_k} \]
is conformal.  Alternatively, we can define $\eta_k$ on $\BA \times J_+$ as the pull-back of the leaf-wise conformal structure by the map $T'_k \times g \colon \BA \times J_+ \to \BA \times J_-$, where $g \colon J_+ \to J_-$ is given by $g(t) = -t$.  Since $T'_k$ is conformal from $\BA_{\epsilon_k}^{\eta_k} \to \BA_{-\epsilon_k}^{\eta_k}$, the leaf-wise conformal structures glue together at $\epsilon_k$, and because we are simply following the first part of the path backward, (\ref{E:bounds final}) holds for all $s \in J_-$, and hence for all $s \in [-h,h]$.

Finally, let
\[ T_k = \phi_{-h+\epsilon_k}^{\BT} T'_k \phi_{-h+\epsilon_k}^{\BT} \colon \BA_h^{\eta_k} \to \BA_{-h}^{\eta_k}. \]
By construction, this map is conformal.  Composing this map with $\phi_{2h}^\BT\colon \BA_{-h}^{\eta_k} \to \BA_h^{\eta_k}$, and from the fact that $(\phi_s^\BT)$ is a local flow on the product, it follows that on $\BA_{-h}^{\eta_k}$ we have
\[ T_k \phi_{2h}^\BT = \phi_{-h+\epsilon_k}^{\BT} T'_k \phi_{-h+\epsilon_k}^{\BT}  \phi_{2h}^\BT = (\phi_{h-\epsilon_k}^\BT)^{-1} \left( T'_k \phi_{2\epsilon_k} \right) \phi_{h-\epsilon_k}^\BT. \]
Therefore, $T_k \phi_{2h}^\BT$ is the conjugate by $\phi_{h-\epsilon_k}^\BT$ of the $k^{th}$ power of a Dehn twist in the core curve of $\BA_{-\epsilon_k}$, and hence is the  $k^{th}$ power of a Dehn twist in the core curve of $\BA_{-h}$, as required.

Since $\mu_s^{\eta_k}$ agrees with the tangent fields from the leaf-wise conformal structures from Corollary~\ref{C: reparameterized} and Lemma~\ref{L:pinch-twist}, it is identically zero outside the middle sub-annulus for each fiber, and hence is identically zero in a neighborhood of $\partial \BA \times J$.  Furthermore, on each piece the flow $\phi_s^\BT$ is a dilation on boundaries of annuli.  These observations complete the proof.
\end{proof}

\subsection{Singular-solv solid tori}
\label{S:singular solv tori}

For the remainder of this section, we assume $M = M_\phi$ with $\phi$ pseudo-Anosov and that $( \pi \colon M \to L\BS^1,(\phi_t))$ is the suspension equipped with the leaf-wise conformal structure from the singular-solv structure, and prove the remaining lemma.

\bigskip

\noindent
{\bf Lemma~\ref{L:good solid torus}} {\em For some $h \geq \tfrac{1}2 \arccosh(\tfrac{\tau_\alpha}2-3)$ and $J = [-h,h]$ there is an embedding $\iota \colon \BT_J \to M$ compatible with the suspension $(\pi \colon M \to L\BS^1,(\phi_t))$ which is disjoint from the singularities.
Furthermore, the induced leaf-wise conformal structure $\zeta$ on $\BT_J$ agrees with the standard one on $\BA_{-h}$, and there exists  $\Psi \colon \BA_h^\zeta \to \BA_{-h}^\zeta$ conformal so $\Psi \circ \phi_{2h}^\BT\colon \BA_{-h} \to \BA_{-h}$ is the $r^{th}$ power of a Dehn twist for some integer $r$.}

\begin{proof} [Proof of Lemma~\ref{L:good solid torus}.]
As we pointed out in \S\ref{S:good solid tori}, $\tau_\alpha(\phi) = d_\alpha(\mathcal L_+,\mathcal L_-)$, is the projection distance to the annulus with core curve $\alpha$ of the stable and unstable laminations for $\phi$.
Since $\tau_\alpha = \tau_\alpha(\phi) \geq 9$, by a result of Rafi \cite{rafi2005characterization} (see \cite{Leininger-Reid-pA} for this specific statement), in any fiber $S_t^\zeta$, there is a Euclidean cylinder neighborhood of (a representative of $\alpha$), and at the balance time for $\alpha$, the modulus of the maximal Euclidean annulus is greater than $\frac{\tau_\alpha}2 - 2$.  By translating, we may assume that the balance time of $\alpha$ is $0$.

Now we let
\[ m_0 = \sqrt{\left( \left\lfloor \sqrt{ \left( \frac{\tau_\alpha}2 - 2\right)^2 - 1} \right\rfloor \right)^2 + 1}.\]
An elementary computation shows
\begin{equation}
\frac{\tau_\alpha}2 - 3 \leq m_0 \leq \frac{\tau_\alpha}2 - 2 \mbox{ and } \sqrt{m_0^2 - 1} \in \mathbb Z.
\end{equation}
It follows that in $S_0^\zeta$ there is a Euclidean annular neighborhood $A'$ of a flat geodesic representative of $\alpha$ of modulus $m_0$.  Another computation shows that the modulus at any time $t \in \mathbb R$ of $\phi_t(A')$ is given by $m_t = \frac{m_0}{\cosh(2t)}$.  Therefore, setting $h = \frac12 \arccosh(m_0)$, we see that $\BA = \phi_{-h}(A')$ is a Euclidean annulus of modulus $1$, and hence by scaling the singular-solv metric if necessary, we can find an isometry $\hat \iota \colon \BS^1 \times [0,1] \to \BA$.  We also note
\[ h = \frac12 \arccosh(m_0) \geq \frac12 \arccosh \left( \tfrac{\tau_\alpha}2-3 \right)  \geq \frac12 \arccosh( 3/2 ) > 0.\]

For $J = [-h,h]$, we define $\iota \colon \BT_J \to M$ by 
\[ \iota((x+iy),s) = \phi_{s+h}(\hat \iota(x+iy)).\]
For each $s$, the image $\iota(\BA_s)$ is a locally isometrically embedded Euclidean annulus (though $\iota$ does not restrict to a local isometry).
The map $\bar \iota \colon J \to L\BS^1$ is the composition of the inclusion $J \to \R$ with the covering $\R \to L\BS^1$ (since we have translated to assume the balance time of $\alpha$ is $0$).  We now establish the key properties of $\iota$.

We first prove that $\iota$ is an embedding.  Since $\BT_J$ is compact and $\iota$ is continuous, it suffices to prove that $\iota$ is injective.  By construction, $\iota$ maps each annulus $\BA_s$, for $s \in J$, injectively to the Euclidean cylinder $\phi_{h+s}(\BA_{-h}) = \phi_s(A')$ which has modulus $\frac{m_0}{\cosh(2s)} \geq 1$.  Suppose $\iota$ is not injective, in which case there are two different values $s \neq s' \in J$ so that $\phi_s(A') \cap \phi_{s'}(A') \neq \emptyset$.  Since $s-s' \neq 0$, $\phi_{s-s'}$ is a nonzero power of the first return map to $\Sigma_{s'}$, hence is a nonzero power of $\phi$.

Now note that two Euclidean annuli in a Euclidean cone surface that intersect either do so in Euclidean sub-annuli or else cross each other transversely. In the latter case the product of the moduli is at most $1$.   If the annuli intersect in a sub-annulus, then the isotopy class of the core curve of $\phi_{s'}(A')$ is sent by a power of a conjugate of $\phi$ to itself, contradicting the fact that $f$ is pseudo-Anosov.  So, the annuli $\phi_s(A')$ and $\phi_{s'}(A')$ must cross each other transversely.  The same is true of the strictly larger maximal Euclidean annuli containing these, whose moduli are therefore strictly greater than $1$.  This contradicts the fact that the product of the moduli is no more than $1$.  Therefore, $\iota$ is injective.

The fact that $\zeta$ agrees with the standard conformal structure on $\BA_{-h} = \BA$ just follows from the fact that $\hat \iota$ is an isometry, hence conformal.  Since $\phi_{t-s} \colon S_s^\zeta \to S_t^\zeta$ is affine for all $t,s$, it follows that $\iota|_{\BA_s} \colon \BA_s \to \phi_s(A')$ is an affine map with respect to the standard structure on $\BA_s = \BS^1 \times [0,1]$ on the domain and the conformal structure $\zeta$ on the image.  In particular, $\iota|_{\BA_h} \colon \BA_h \to \phi_{h}(A')$ is an affine map from a Euclidean annulus of modulus $1$ to another Euclidean annulus of modulus $1$ of the same area.  We can therefore isometrically parameterize $\phi_h(A')$ by $\BS^1 \times [0,1]$ so that with respect to these coordinates, the map $\iota|_{\BA_h}$ is given by
\begin{equation} \label{E:affine twist in coordinates} x+iy \mapsto (x+ry) + iy
\end{equation}
for some $r \in \mathbb R$.  The norm of the Beltrami coefficient of this map is $\frac{|r|}{\sqrt{r^2+4}}$, while on the other hand the map is $e^{4h}$--quasi-conformal (since the map $\phi_h \colon S_{-h}^\zeta \to S_h^\zeta$ is a Teichm\"uller map), and hence the norm of the Beltrami coefficient is also given by
\[ \frac{e^{4h}-1}{e^{4h}+1} = \frac{\sinh(2h)}{\cosh(2h)} = \frac{\sinh(\arccosh(m_0))}{\cosh(\arccosh(m_0))} = \frac{\sqrt{m_0^2-1}}{m_0}.\]
Setting these quasi-conformal dilatations equal, we have
\[ \frac{|r|}{\sqrt{r^2+4}} = \frac{\sqrt{m_0^2-1}}{m_0}, \]
and solving we see that $r = \pm 2 \sqrt{m_0^2-1} \in 2 \mathbb Z$.  Thus if we let $\Psi \colon \BA_h^\zeta \to \BA_{-h}^\zeta$ be a conformal map, the composition $\Psi \circ \phi_{2h}^\BT$ is given by the formula in (\ref{E:affine twist in coordinates}) and so is the $r^{th}$ power of a Dehn twist, completing the proof.
\end{proof}

\bibliography{refs.bib}
\bibliographystyle{amsalpha}
\end{document}